\documentclass[12pt,a4paper,twoside,final,notitlepage, reqno]{article}

\usepackage[english]{babel}
\usepackage[latin1]{inputenc}
\usepackage[a4paper,left=2.5cm,right=2.5cm]{geometry}
\usepackage{amssymb}
\usepackage{amsthm} 
\usepackage{amsmath}
\usepackage{graphics,graphicx}
\usepackage[usenames, dvipsnames]{color}
\usepackage{hyperref}

\usepackage{booktabs}
\usepackage{subfigure} 
\usepackage{siunitx}

\graphicspath{{figs/}}

\makeatletter
\def\input@path{{figs/}}
\makeatother

\theoremstyle{definition}

\newtheorem{proposition}{Proposition}

\theoremstyle{definition}
\newtheorem{definition}{Definition}

\theoremstyle{definition}
\newtheorem{remark}{Remark}

\providecommand{\keywords}[1]  {\textbf{Keywords:} #1}

\def\Tau{{\boldsymbol{\mathfrak M}}}
\def\ind{I}
\def\proj{p}
\renewcommand \ker[1] { \mathop{\rm Ker}\nolimits \left(#1\right) }
\newcommand \ran[1] { \mathop{\rm Ran}\nolimits\left(#1\right) }
\renewcommand{\div}{\mathop{\rm div}\nolimits}
\def\Delt{{{\scriptstyle \Delta} t}}
\newcommand \tsp[1] { \,^{T}#1 }
\newcommand \psinv[1] { #1^{ \widetilde{-1} } }
\newcommand{\rr}{\mathbb{R}}
\newcommand{\n}{\mathbf{n}}
\newcommand{\oo}{\Omega}        
\newcommand{\HH}{H}             
\newcommand{\TT}{T}             
\newcommand{\am}{\chi}
\newcommand{\cm}{c}
\newcommand{\vm}{v}
\newcommand{\VM}{V}
\newcommand{\uu}{u}
\newcommand{\UU}{U}
\newcommand{\ion}{I_{ion}}
\newcommand{\iapp}{I_{st}}
\newcommand{\ww}{\mathbf{w}}
\newcommand{\WW}{W}
\newcommand{\G}{\sigma}              
\newcommand{\Ge}{\sigma_e}           
\newcommand{\Gi}{\sigma_i}
\newcommand{\Gt}{\sigma_\TT}           
\newcommand{\Gun}{\overline{\sigma}_1}
\newcommand{\Gv}{\overline{\sigma}_e}

\newcommand{\mass}{M}
\newcommand{\stiff}{S}
\newcommand{\rst}{\Pi}
\renewcommand{\prec}{P}
\newcommand{\sysmat}{\Lambda}

\usepackage{authblk}

\title{
  \vspace{-1.5cm}
  \bf{\Large{
      Preconditioning the bidomain model 
      \\[5pt]
      with almost linear complexity
    }}
}

\author[1]{Charles Pierre \thanks{charles.pierre@univ-pau.fr}}

\affil[1]{
  Laboratoire  de Math\'ematiques et de leurs Applications,
  UMR CNRS 5142, \protect \\
  Universit\'e de Pau et des Pays de l'Adour, France.
}

\usepackage{fancyhdr}

\begin{document}
\date{30 September, 2011}
\maketitle
\begin{abstract}
  The bidomain model is widely used in electro-cardiology to simulate  spreading of excitation in the myocardium and electrocardiograms.
  It consists of a system of two parabolic reaction diffusion equations coupled with an ODE system. Its discretisation displays an ill-conditioned system matrix to be inverted at each time step: simulations based on the bidomain model therefore are associated with high computational costs. In this paper we propose a preconditioning for the bidomain model in an extended framework including a coupling with the surrounding tissues (the torso).
The preconditioning is based on a formulation of the discrete problem that is shown to be symmetric positive
semi-definite. 
A block $LU$ decomposition of the system together with a heuristic approximation (referred to as the monodomain approximation) are the key ingredients for the preconditioning definition.
Numerical results are provided for two test cases:
a 2D test case on a realistic slice of the thorax based on a segmented heart medical image geometry, a 3D test case involving a small cubic slab of tissue with orthotropic anisotropy.
The analysis of the resulting computational cost (both in terms of CPU time and of iteration number) shows an almost linear complexity with the problem size, i.e. of type $n\log^\alpha(n)$ (for some constant $\alpha$) which  is optimal complexity for such problems.
\end{abstract}
\vspace{20pt}
\noindent
\vspace{2cm}
\section{Introduction}
\label{sec:intro}
The bidomain model \cite{tung-bid,krassowska-neu-93,ambr-colli-sav-00,colli-savare-02,veneroni-06,pullan-book-2005,clayton-2010} is up to now the most physiologically founded model to describe the heart electrical activity. 
The bidomain model is here considered in an extended version referred to as the \textit{coupled heart and torso bidomain model}. It includes a coupling of the cardiac electrical activity with the surrounding tissue  electrical activity, allowing in particular electrocardiogram simulations.
\\
The bidomain model mathematical formulation is composed of a system of two PDEs (parabolic reaction diffusion equations) describing the evolution of two potentials: the \textit{intra-} and \textit{extra-cellular potentials} within the myocardium. This system is coupled with a set of ODEs modelling the kinetic of ionic transfer across the cellular membrane.
\\
The discretisation of the bidomain model displays an ill conditioned system matrix to be inverted at each time step. 
This is essentially due to the nature of the model.  
Two reasons are raised for this. 
The bidomain model can be formulated as a degenerate system of two coupled parabolic equations \cite{colli-savare-02}, which degeneracy causes ill-conditioning. 
Another formulation of the bidomain model, made up of a single scalar semi-linear parabolic equation, is studied in \cite{2yves-pierre-jnla07}. This formulation involves a non-local  operator of second order in space, referred to as the \textit{bidomain operator}.
The bidomain operator is defined as the harmonic mean between two elliptic operators.
The non-locality of the bidomain operator generates high computational costs. 
\\
On top of this structural ill-conditioning, the physical features of the modelled phenomena (because of fast and sharp space and time variations of potential: namely transmembrane potential wave fronts) necessitates to resort to fine space and time grids.
Ill conditioning together with fine meshes imply very high computational costs for the bidomain model simulations that remain challenging for 3D realistic settings.
For this, many efforts were devoted to the reduction of this cost, see e.g. \cite{colli-pavarino-2004,bourgault-ethier,colli-SIAM06,lines-tveito-2006,lines-tveito-2007,perego-2010}.
\\[3pt]\indent
Few papers are dealing with the preconditioning of the bidomain model.
In \cite{pavarino-schwarz-2008} Pavarino and Scacchi proposed a preconditioner designed to a parallel implementation of the bidomain model. In \cite{nobile-precond-2009} Gerardo-Giorda et al. introduced a very interesting preconditioning strategy discussed deeper on at the end of this section.
\\
The aim of this paper is to define a general preconditioning for the bidomain system of equations.
This preconditioning is based on two simple ideas (detailed hereafter in this section): an algebraic block-$LU$ factorisation together with a  heuristic approximation.
For its implementation in practise, it only remains to define two local block preconditioners for two matrices: obtained by discretising an elliptic and a parabolic type equations respectively.
A wide class of preconditioners for such problems already has been developed, either sequential or parallel, with available implemented versions (see e.g. \cite{hackbusch-book-85,saad,ccg-96,benzi-2002,hackbusch-2001-Hmat}, details follow).
We actually can resort to any of these preconditioners to embed it into the bidomain model preconditioning here presented.
In this sense, our preconditioning framework provides a lifting from preconditioners for elliptic problems to preconditioners for the bidomain model.
\\
The natural question raised by this is: ``can we recover the (already available) high performances of elliptic problem preconditioners for the bidomain equations ?''.
This question is here addressed from the point of view of complexity.
Let $A$ denote a sparse matrix with size $n$ obtained by discretising an elliptic equation. Optimal complexity to perform $X\mapsto A^{-1}X$ is in $O(n\log(n)^\alpha)$ ($\alpha$ constant) referred to as almost linear complexity (developments on complexity matters are given in Sec. \ref{sec:prec3}).
Optimal complexity has been obtained for elliptic problems for instance using multi-grid approaches \cite{hackbusch-book-85,ccg-96} or hierarchical matrix factorisations \cite{hackbusch-2001-Hmat,lars-hackbusch-2002,lars-hackbusch-2003,lars-leborne-2008}.
In this paper we numerically prove that almost linear complexity can be reached for the bidomain model embedding a hierarchical Cholesky decomposition into our general bidomain model preconditioning.
\\[3pt]\indent
Several (equivalent) mathematical formulations of the bidomain model have been proposed: we refer to \cite{clayton-2010} for a comprehensive review.
The bidomain model can be set as a system of two coupled degenerate parabolic equations: this formulation has been used to prove existence of solutions in \cite{colli-savare-02,boulakia-2008} and numerically used e.g. in \cite{sanfelici02,colli-taccardi-05,pavarino-schwarz-2008}.
A second formulation involves a coupled parabolic-elliptic system of two equations.
This formulation has been widely studied either for theoretical or numerical purposes:
either using non-symmetric versions (see for instance \cite{nobile-precond-2009,ABKP-2010}) or a self-adjoint positive semi-definite version studied in \cite{2yves-pierre-jnla07}.
We consider here a general discretisation of the self-adjoint formulation. 
This discrete formulation of the bidomain model is here shown to be symmetric positive semi-definite: this property holds including the coupling of the heart with the surrounding tissues.
This discrete formulation of the bidomain model has already been used e.g. in \cite{belhamadia-bourgault-2008,boulakia-2010}.
\\
Embedding the strong structural properties of the bidomain model (i.e. symmetry and positivity) at the discrete level is quite natural and should provide an efficient implementation.
We personally experienced the difference between the symmetric positive formulation here adopted and the non-symmetric one in \cite{ABKP-2010}. 
A gain in CPU time of factor more than 5  was made with the symmetric positive version and for a similar resolution strategy.
\\[3pt]\indent
Let us now detail the general preconditioning strategy.
It relies on the symmetric positive semi-definite formulation of the coupled heart and torso bidomain model.
Various  space discretisations (including classical Lagrange $P^k$ finite elements or various finite volume techniques) can be considered.
For simplicity we adopted  here  an Euler semi-implicit time discretisation but the technique generalises to more sophisticated time schemes.
Once discretised, this formulation involves the inversion of one system matrix (symmetric positive semi-definite) per time step. 
The two following points are used to precondition the system matrix.
\begin{itemize}
\item[\textbf{1-}]   \textbf{$LU$ factorisation.} The system matrix displays a $2\times 2$ block structure that can be factorised into a block-$LU$ form.
\item[\textbf{2-}]   \textbf{Monodomain model heuristic.} Among the blocks of the $LU$ factorisation, all blocks have a simple definition (they are sparse and do not lead to computational difficulties) except one block. This block is shown to be symmetric positive definite and to 
be the sum of a mass matrix and of a
\textit{discrete bidomain operator} (discrete analogue of the   bidomain operator mentioned earlier on) that  is shown to be the harmonic mean between two stiffness matrices. This block, that is not sparse, is not computed but approximated using the \textit{monodomain model} approximation detailed below.
\end{itemize}
The monodomain model approximation basically consists in approximating the \textit{bidomain operator} in \cite{2yves-pierre-jnla07} (the harmonic mean between two diffusion operators) by a simple diffusion operator. 
The monodomain model can provide an accurate approximation of the bidomain model \cite{clem-nenon-horac-2004,colli-taccardi-05,dube-potse-06,PRB1-2010}.
It has been shown in  \cite{PRB1-2010} that a monodomain model could provide activation time mappings in complex situations with 1\% of relative error as compared to the bidomain model predictions.
The diffusivity tensor for the monodomain model approximation will here be set to the harmonic mean of the intra- and extra-cellular conductivity tensors.This approximation is heuristic, it is exact in dimension 1 and in case of equal anisotropy ratio between the intra- and extra-cellular media.
%
%
%
\\[3pt]\indent
In a recent paper \cite{nobile-precond-2009}, Gerardo-Giorda et al. introduced a preconditioner for the bidomain model also based on a monodomain model heuristic approximation and on a lower block triangular approximation.
Let us point out the differences between these two papers.
The $LU$ factorisation presented here should provide more efficient algorithms than the lower bock-triangular approximation since this factorisation is exact.
The formulation in \cite{nobile-precond-2009} is based on a non-symmetric formulation whereas we here considered a symmetric  positive semi-definite system matrix.
We then can benefit from symmetry and positivity properties in terms of computational efficiency, for instance resorting to a  conjugate gradient linear solver.
A draft of quantitative comparison between these two preconditioning is made in the conclusion section \ref{sec:conc}.
\\[3pt]\indent
The paper is organised as follows.
The coupled heart and torso bidomain model is stated in Sec. 2. Its numerical discretisation follows in Sec. 3.
In Sec. 4 are stated and proved the mathematical properties of the discretised bidomain problem system matrix: it is  shown to be symmetric  positive semi-definite, its $LU$ block factorisation is then analysed.
The general preconditioning of the bidomain model is defined in Sec. 5, 
sub section \ref{sec:prec3} is devoted to its practical implementation.
Numerical results are in Sec. 6. 
The two test cases are presented in 6.1. 
The complexity of the preconditioned system matrix inversion is numerically studied in Sec. 6.2. Results are discussed in the conclusion section 6.3.
\section{Bidomain model of the heart embedded in the torso}
\label{sec:model}
Let us denote by $\oo$ and $\HH$ two bounded open subsets such that
$\HH\subset \oo\subset \rr^d$ with $d=2,~3$ and with smooth
boundaries. 
We moreover assume that $\partial\oo \cap\partial\HH = \emptyset$:  $\oo$ represents a thorax and $\HH$ the region occupied
by the heart (assumed fixed here). 
We also consider
$\TT:=\oo-\overline{\HH}$ that 
will be referred to as the \textit{torso}, see Fig. \ref{fig:test-case}.
We denote $Q$, $Q_\HH$ and $Q_\TT$ the  time-space cylinders 
$\rr^+\times\oo$, $\rr^+\times\HH$ and $\rr^+\times\TT$ respectively.

Two potential fields will be involved, the transmembrane
potential $\vm:~Q_\HH\mapsto \rr$ and the potential 
$\uu:~Q\mapsto \rr$. When  restricted to $\HH$ (\textit{resp.} to $\TT$),
the potential $\uu$ is referred to the extra-cellular potential
(\textit{resp.} extra-cardiac potential).
The  transmembrane
potential  $\vm=u_i-\uu_{|\HH}$ is the difference between an intra-cellular potential $u_i:~Q_\HH\mapsto \rr$ and the
extra-cellular potential $\uu_{|\HH}$; the intra-cellular potential will not be considered in the following mathematical formulation of the problem.

The heart has a fibrous organisation implying anisotropic electrical
conductivities. The cardiac fibres rotate around the ventricular cavities, see Fig. \ref{fig:test-case}. The fibres remain tangent to the cardiac boundaries. This anisotropy is taken into account
by introducing in $\HH$ two  tensors $\Gi$ and $\Ge$. Introducing the 4
conductivity parameters  $g^l_{i,e}$, $g^t_{i,e}$, they read as follows:
\begin{align*}
  \Gi(x)=\text{Diag}(g_i^l,g_i^t),\quad \Ge(x)=\text{Diag}(g_e^l,g_e^t),
\end{align*}
in a moving system of coordinates whose principal orientation is given by the fibre
orientation at point $x$. Of course, when written in a fixed basis, these tensors  no longer are diagonal.
Physically, the parameters $g^l_{i,e}$ and $g^t_{i,e}$  are the electrical conductivities longitudinally and transversely to the fibre direction (subscript $l$ and $t$) and relatively to the intra- or extra-cellular media (index $i$ or $e$) respectively.

The torso region $\TT$ is assumed to have an isotropic but heterogeneous electrical conductivity. We define in $\TT$ the conductivity tensor $\Gt(x)=k(x)Id$ where the conductivity $k:\TT\mapsto \rr$ basically is piecewise constant on the different organs considered in $\TT$.

The torso model consists in:
\begin{equation}
  \label{eq:torso}
  \left\{ 
  \begin{aligned}
     \div(\Gt(x) \nabla \uu) &=0, \quad(t,x)\in &Q_\TT,
     \\
     \nabla \uu \cdot \n&=0 \qquad\text{on}\quad &\partial\oo,
    \end{aligned}
    \right.
\end{equation}
where $\n$ denotes the outward unit normal to $\partial\oo$.

In the heart region, the bidomain model is composed of the three following equations in $\HH$, for $(t,x)\in Q_\HH$:
\begin{equation}
  \label{eq:bid}
  \left\{ 
  \begin{aligned}
    \div((\Gi(x)+\Ge(x))\nabla \uu) &= -\div(\Gi(x)\nabla \vm) ,
    \\
    \am\left(
    \cm \partial_t \vm + \ion(\vm,\ww) -\iapp(t,x)
    \right) &= \div(\Gi(x)\nabla (\uu+\vm)),
    \\
    \partial_t \ww&= g(\vm,\ww).
  \end{aligned}
  \right.
\end{equation}
In the second equation, $\cm$ denotes the cell membrane surface capacitance, 
$\am$ is the ratio of cell membrane surface per
unit volume,
$\iapp:~ Q_\HH\mapsto\rr$ is the stimulation current (source term). 
$\ion(\vm,\ww)$ (reaction term) denotes the surface ionic current distribution on the membrane. The gating variable $\ww~:Q_\HH\mapsto\rr^p$ characterises the state of the cell membrane, its evolution is ruled by the ODE system in the third equation. 
The definitions of $\ion$ and of $g$ are fixed by the chosen ionic model in Sec. \ref{sec:ionic-model}. 
\\
Equations \eqref{eq:bid} are coupled with the torso model \eqref{eq:torso} 
with the following coupling condition:
\begin{equation}
  \label{eq:bid-coupling}
  \text{on}\quad   \partial \HH:\quad 
  \left\{ 
  \begin{aligned}
    \uu_{\vert \HH}=\uu_{\vert \TT} ~,\quad
    \Ge(x)\nabla \uu_{\vert \HH} \cdot \n&=
    \Gt(x)\nabla \uu_{\vert \TT} \cdot \n,
    \\
    \Gi\nabla \uu_{\vert \HH} \cdot \n+ \Gi\nabla \vm \cdot \n
    &=0. 
  \end{aligned}
  \right.
\end{equation}
where $\n$ denotes the outward unit normal to $\partial\HH$.

The model is closed by  imposing  initial conditions on $\vm$ and $\ww$,
\begin{equation}
    \label{eq:bid-ic}
    \vm(0,x)=\vm_0(x), \quad \ww(0,x)=\ww_0(x), \quad x\in \HH.
\end{equation}
Clearly, the potential field $\uu$ is defined up to an additive
constant.
We therefore impose the  normalisation condition for all time $t>0$:
\begin{equation}
  \label{eq:u-normalized}
  \int_\oo \uu(t,\cdot)dx = 0 .
\end{equation}

\subsection{Weak formulation}
We introduce the  tensor $\Gun$ on $\oo$:
\begin{displaymath}
  \Gun(x)=\left\{
    \begin{aligned}
      \Gi(x)+\Ge(x),\quad x\in\HH
      \\
      \Gt(x),\quad x\in\TT
    \end{aligned}
  \right.
  .
\end{displaymath}

The weak formulation of the bidomain model \eqref{eq:torso}, \eqref{eq:bid}, \eqref{eq:bid-coupling} is the following:
$\forall \psi\in H^1(\oo)$, $\forall \phi\in H^1(\HH)$,
\begin{align}
  \label{eq:impl-weak}
  \left\{
    \begin{aligned}
      \int_\oo\Gun\nabla \uu\cdot\nabla \psi dx
      +
      \int_\HH\Gi\nabla \vm\cdot\nabla \psi dx&=0,
      \\
      \am\cm\partial_t \int_\HH \vm\phi dx + \am\int_\HH (\ion(\vm,\ww)-\iapp(x,t))\phi dx
      &=
      -\int_\HH\Gi\nabla (\uu+\vm)\cdot\nabla \phi dx      ,
    \end{aligned}
  \right.
\end{align}

The first equation in \eqref{eq:impl-weak} is obtained by  multiplying \eqref{eq:torso} and the first equation in  \eqref{eq:bid} by a test function  $\psi\in H^1(\oo)$, by integrating on  $\oo$ and by using  the coupling conditions \eqref{eq:bid-coupling} and the boundary condition  \eqref{eq:torso}. 
The second equation in \eqref{eq:impl-weak} is obtained by  multiplying the second equation in  \eqref{eq:bid} by a test function  $\phi\in H^1(\HH)$, by integrating on  $\HH$ together with \eqref{eq:bid-coupling}.
\subsection{Case of an isolated heart}
\label{subsec:isol-case}
We here address the  case where the heart is considered as
isolated from the surrounding tissues. In this case we have $\HH=\oo$ and
$\TT=\emptyset$.  Equations \eqref{eq:bid} only are considered and the
coupling conditions \eqref{eq:bid-coupling} are replaced by zero flux boundary
conditions on $\partial\HH$ for $\vm$ and $\uu$.
\section{Implementation}
\label{sec:methods}
For  simplicity,
temporal discretisation is fixed to a semi implicit Euler scheme:
implicit for the diffusion and explicit on the reaction. 
Extensions to other time schemes is possible as discussed in remark \ref{rem:time-disc}.

The implementation strategy is similar for various space discretisations
including  $P^k$ Lagrange finite elements or finite volume scheme such as
the CVFE scheme
 (Control Volume Finite Element, see e.g. \cite{cai_etal_1991}) 
or such as the DDFV scheme in \cite{ABKP-2010}. 
Assumptions  $(H1)$ and $(H2)$ on the space discretisation  are
detailed in Sec. \ref{sec:imp-set} whereas the numerical scheme
itself is presented in Sec. \ref{sec:imp-scheme}.

\subsection{Settings}
\label{sec:imp-set}
Let us consider  a mesh $\Tau$ of  $\oo$ and 
  a mesh $\Tau_\HH$ of the cardiac region $\HH$: 
we assume that $\Tau_\HH$ is a sub mesh of $\Tau$, that is to say that
all elements (or cells or control volumes) of $\Tau_\HH$ also are
elements of $\Tau$.

Relatively to the considered space
discretisation, let us denote by $\rr^\Tau$,  $\rr^{\Tau_\HH}$ the set of
discrete functions attached to these two meshes. 
Their dimensions are  denoted $N$  and $N_\HH$ respectively.
A ``\textit{natural}'' basis usually is provided for $\rr^\Tau$ and
$\rr^{\Tau_\HH}$,  denoted $(U_i)_{1\le i\le N}$ and
$(U^\HH_i)_{1\le i\le N_\HH}$ respectively. In the case of $P^k$ finite element
methods, these functions simply are the standard $P^k$ Lagrange basis
functions. Considering these basis  induces an isomorphism between
$\rr^\Tau$ and $\rr^N$ and between $\rr^{\Tau_\HH}$ and  $\rr^{N_\HH}$.
A discrete function $U=\sum_{i=1}^N c_i U_i$ will be
considered either as a real function or as the real vectors $(c_i)_{i\le
  1\le N}$.
Using these identifications, the  canonical Euclidian structures
on $\rr^N$ and $\rr^{N_\HH}$ extend to $\rr^\Tau$ and 
$\rr^{\Tau_\HH}$. We denote by 
$(\cdot,\cdot)_{\Tau}$ and 
$(\cdot,\cdot)_{\Tau_\HH}$ the associated scalar products.
\\
We make the following first assumption on the space discretisation method:
\begin{itemize}
\item [$(H1)$] for all $i$, $1\le i\le N_\HH$: $U_{i\vert \HH} = U^H_i$
  (where $U_{i\vert \HH}$ denotes the restriction of the function $U_i$
  to $\HH$).
\end{itemize}
In the case of the $P^k$ finite element methods, this first assumption is
true modulo a reordering of the basis functions $(U_i)_{1\le i\le N}$.
Assumption $(H1)$ allows us to define the restriction operation:
\begin{equation}
  \label{eq:def-Pi}
  \rst~:\quad U = \sum_{i=1}^N c_i U_i \in \rr^\Tau
  \mapsto U_{\vert  \HH} = \sum_{i=1}^{N_\HH} c_i U^\HH_i \in
  \rr^{\Tau_\HH}.
\end{equation}
Equivalently, $\rst$ can be seen as a simple truncation operation:
\begin{equation*}
   \rst~:\quad U = (c_i)_{1\le i\le N} \in \rr^\Tau
  \mapsto U_{\vert  \HH} =(c_i)_{1\le i\le N_\HH} 
  \in  \rr^{\Tau_\HH},
\end{equation*}
following the above described identification between $\rr^\Tau$ and
$\rr^N$ and between $\rr^{\Tau_\HH}$ and
$\rr^{N_\HH}$. The transpose mapping $\tsp{\rst}$ for $\rst$ is:
\begin{equation*}
  \label{eq:def-TPi}
  \tsp{\rst}~:\quad
  U= \sum_{i=1}^{N_\HH} c_i U^\HH_i \in \rr^{\Tau_\HH}
  \mapsto \sum_{i=1}^{N_\HH} c_i U_i \in \rr^\Tau.
\end{equation*}
We point out that in this discrete setting $\tsp{\rst}$ does not match
the prolongation by zero outside $\HH$. The following property will be useful:
\begin{equation}
  \label{eq:h4}
  \rst \tsp{\rst} = id_{\rr^{\Tau_\HH}}.
\end{equation}
Let us introduce the mass matrices $\mass$, $\mass_\HH$ and the stiffness
matrices $\stiff_1$, $\stiff_i$ so that:
\begin{align*}
  \forall ~~U_1,U_2\in\rr^\Tau&: ~~\int_\oo U_1U_2dx = (\mass U_1,U_2)_\Tau,~~
  \int_\oo \Gun \nabla U_1\cdot\nabla U_2dx = (\stiff_1 U_1,U_2)_\Tau
  \\
  \forall ~~V_1,V_2\in\rr^{\Tau_\HH}&: ~~\int_\HH V_1V_2dx = (\mass_\HH V_1,V_2)_{\Tau_\HH},~~
  \int_\HH \Gi \nabla V_1\cdot\nabla V_2dx = (\stiff_i V_1,V_2)_{\Tau_\HH}
\end{align*}
The second assumption on the space discretisation is the following:
\begin{itemize}
\item [$(H2)$] Let us denote $\ind_\oo$ and $\ind_\HH$ the characteristic functions of
  $\oo$ and $\HH$ respectively (constant functions equal to one):
  \begin{equation}
    \label{eq:def-ind}
    \ind_\oo\in\rr^\Tau~,\quad \ind_\HH\in\rr^{\Tau_\HH}
    .
  \end{equation}
\end{itemize}
Assumption $(H2)$ is related with the considered boundary conditions here: homogeneous Neumann on $\partial\Omega$ and transmission conditions on $\partial\HH$. 
It implies that the stiffness matrices $\stiff_1$, $\stiff_i$ (that are symmetric positive semi-definite) have for kernels the one dimensional spaces $\ind_\oo\rr$ and  $\ind_\HH\rr$ respectively.

\subsection{Scheme statement}
\label{sec:imp-scheme}
The three unknowns $\vm$, $\uu$ and $\ww$ of the
(continuous) bidomain model  are represented by the discrete functions 
$\UU\in\rr^\Tau$,  $\VM\in\rr^{\Tau_\HH}$ and 
$\WW\in[\rr^{\Tau_\HH}]^p$. 
\\
We have for all test function $\Psi\in\rr^\Tau$:
\begin{displaymath}
  \int_\HH \Gi \nabla V\cdot\nabla \Psi dx =
  (\stiff_i V, \rst \Psi)_{\Tau_\HH} = (\tsp{\rst} \stiff_i V, \Psi)_{\Tau}
\end{displaymath}
Discretisation of \eqref{eq:impl-weak} thus is:
\begin{equation}
  \label{eq:bid-disc-a}
  \left\{
    \begin{aligned}
      \stiff_1 \UU^{n+1} 
      +
      \tsp{\rst} \stiff_i \VM^{n+1}&=0,
      \\
      \am\cm\mass_\HH \dfrac{\VM^{n+1}-\VM^{n}}{\Delt}
      +
      \am \mass_H\left(
        \ion(\VM^n,\WW^n)-\iapp^n
      \right)
      &=
      -\stiff_i \rst\UU^{n+1} -\stiff_i \VM^{n+1}
    \end{aligned}
  \right.
  .
\end{equation}
We introduce the positive parameter $\gamma$:
\begin{equation*}
  \gamma := \am\cm/\Delt.
\end{equation*}
\begin{proof}[\textbf{Resolution algorithm}]
The complete bidomain model
\eqref{eq:torso}~\eqref{eq:bid}~\eqref{eq:bid-coupling}
is numerically solved applying the following three operations at each time step.
\\
Being given
$\VM^n\in\rr^\Tau$ and 
$\WW^n\in[\rr^{\Tau_\HH}]^p$:
\begin{itemize}
\item[\textbf{Step 1.}]
  Compute the right hand side $Y$:
  \begin{displaymath}
    Y :=   
    \left[
    \begin{array}{l}
      0
      \\
      \mass_\HH
      \left(
        \gamma\VM^{n} -
        \am (\ion( \VM^n,\WW^n)-\iapp^n)      
      \right)
    \end{array}
  \right. .
  \end{displaymath}
\item[\textbf{Step 2.}] 
  find the solution $X=\tsp{[\UU^{n+1},\VM^{n+1}]}$ to $\sysmat X = Y$ with
  \begin{equation}
    \label{eq:system}
    \sysmat:=
    \left[
      \begin{array}{ccc}
        \stiff_1 &~& \tsp{\rst}\stiff_i
        \\
        \stiff_i\rst &~& \gamma\mass_\HH + \stiff_i
      \end{array}
    \right]
    \quad \quad 
    \text{that satisfies}
    \quad 
    \int_\oo\UU^{n+1}dx = 0.
  \end{equation}
\item[\textbf{Step 3.}]
  Update the gating variable by computing $\WW^{n+1}$ according to 
  the third equation in equation \eqref{eq:bid}.
\end{itemize} 
\end{proof}

This paper is devoted to Step 2 only. Proposition \ref{prop:wp} states that step 2 is well posed.

\section{Properties and $LU$ factorisation of the system matrix $\sysmat$}
\label{sec:sysmat}

Let us precise that
$\stiff_1:\rr^{\Tau}\mapsto \rr^{\Tau}$
and that 
$\stiff_i:\rr^{\Tau_\HH}\mapsto \rr^{\Tau_\HH}$. Then, 
$\sysmat: \rr^{\Tau}\times \rr^{\Tau_\HH}\mapsto \rr^{\Tau}\times \rr^{\Tau_\HH}$.
\begin{proposition}
  \label{prop:wp}
  The system matrix  $\sysmat$ is symmetric positive semi-definite with kernel 
  $\ker{\sysmat}=\ind_\oo\rr\times\{0\}$.
By symmetry $\sysmat$ has for range $\ran{\sysmat}= \ind_\oo^\perp\times  \rr^{\Tau_\HH} $.
  For all $(Y_1,Y_2)\in \ind_\oo^\perp\times
  \rr^{\Tau_\HH}$, there 
  exists a unique $(U,V)\in\rr^\Tau\times \rr^{\Tau_\HH}$
  such that 
  \begin{equation}
    \label{eq:prop-syslin}
    \sysmat
    \left[
      \begin{array}{c}
        U\\V
      \end{array}
    \right]
    =
    \left[
      \begin{array}{c}
        Y_1\\Y_2
      \end{array}
    \right]
    \quad \quad 
    \text{and}
    \quad 
    \int_\oo Udx=0.
  \end{equation}
\end{proposition}

The resolution of step 2 in the resolution algorithm proceeds in two steps: first find a solution $\tsp{[X_1,X_2]}$, then normalise $X_1$. 
We now focus on the first step. 

\begin{definition}[Pseudo-inverses $\psinv{\stiff_1}$ and $\psinv{\stiff_i}$]
  \label{def:pseudo-inv}
  The stiffness matrices $\stiff_1$ and
$\stiff_i$  are isomorphisms on $\ind_\oo^\perp=\ran{\stiff_1}$ and on $\ind_\HH^\perp=\ran{\stiff_i}$ respectively. We introduce their  pseudo inverses $\psinv{\stiff_1}$ and 
$\psinv{\stiff_i}$: they  are equal to the  inverse of $\stiff_1$, $\stiff_i$  on $\ind_\oo^\perp$, $\ind_\HH^\perp$ respectively and equal to 0 on $\ind_\oo\rr$,  $\ind_\HH\rr$ respectively.
\\
Considering $\proj_\oo$ (\textit{resp}. $\proj_\HH$) the
orthogonal projection of $\rr^\Tau$ on $\ind_\oo^\perp$
(\textit{resp.} of  $\rr^{\Tau_\HH}$ on
$\ind_\HH^\perp$), we have:
\begin{displaymath}
  \psinv{\stiff_1} \stiff_1 = \stiff_1\psinv{\stiff_1} = \proj_\oo~,\quad
  \psinv{\stiff_i} \stiff_i = \stiff_i\psinv{\stiff_i} = \proj_\HH~.  
\end{displaymath}

\end{definition}

\begin{proposition}
  \label{prop:LU}
  We have the  block decomposition $\sysmat = LU$ with:
\begin{equation}
  \label{eq:LU}
  L:=\left[
    \begin{array}{ccc}
      \stiff_1 &~&  0
      \\
      \stiff_i \rst &~& 
      K
    \end{array}
  \right] ,
  \quad 
  U:=\left[
    \begin{array}{ccc}
      id_{\rr^{\Tau}}&~&  \psinv{\stiff_1}\tsp{\rst}\stiff_i
      \\
      0 &~&  id_{\rr^{\Tau_\HH}}
    \end{array}
  \right] ,
\end{equation}
The matrix $K$ is symmetric, positive definite, it is defined by:
\begin{equation}
  \label{eq:LU2}
  K := \gamma\mass_\HH + \stiff_i -  \stiff_i\rst\psinv{\stiff_1}\tsp{\rst}\stiff_i.  
\end{equation}
\end{proposition}

\begin{remark}[About the matrix $K$]
  \label{rem:K}
  Let us consider the tensor
  \begin{displaymath}
    \Gv(x)=\left\{
      \begin{aligned}
        \Ge(x),\quad x\in\HH
        \\
        \Gt(x),\quad x\in\TT
      \end{aligned}
    \right.
    ,
  \end{displaymath}
  and denote $\stiff_e$ the associated stiffness matrix.
  Since $\stiff_1$ and $\stiff_e$ have the same range $\ind_\oo^\perp$, one can define the pseudo-inverse $\psinv{\stiff_e}$ for $\stiff_e$ with the same meaning as for $\stiff_1$.

  The matrix $K$ in \eqref{eq:LU2} can be rewritten as
  \begin{displaymath}
    K = \gamma\mass_\HH + 
    \left(\stiff_i^{-1} + \rst \stiff_e^{-1}\tsp{\rst}\right)^{-1}.
  \end{displaymath}
  where all inverses are pseudo-inverses. 
  This equality is precisely stated and proved in the proof of  proposition \ref{prop:LU}. 
  \\
  It is interesting to notice that the second term appears as the ``\textit{harmonic mean}'' between the stiffness matrices $\stiff_i$ and $\stiff_e$.
  At the discrete level, this is a transposition of the
  ``\textit{bidomain operator}'' as defined in \cite{2yves-pierre-jnla07} that was
  introduced as the harmonic mean between two diffusion operators.
\end{remark}

\begin{proposition}
  \label{prop:LU2}
  $L$ has a pseudo inverse $\psinv{L}$ in the
following sense:
\begin{displaymath}
  L \psinv{L} = \psinv{L} L = 
  \left[
    \begin{array}{ccc}
      \proj_\oo &~&  0
      \\
      0 &~& id_{\rr^{\Tau_\HH}}
    \end{array}
  \right],
\end{displaymath}
$U$ is invertible, $U^{-1}$ and  $\psinv{L}$ are given by:
\begin{equation}
  \label{eq:LU-inv}
    \psinv{L}=\left[
    \begin{array}{ccc}
      \psinv{\stiff_1} &~&  0
      \\
      -K^{-1} \stiff_i\rst \psinv{\stiff_1}&~& 
      K^{-1}
    \end{array}
  \right] ,
  \quad 
  U^{-1}=\left[
    \begin{array}{ccc}
      id_{\rr^{\Tau}}&~&  -\psinv{\stiff_1}\tsp{\rst}\stiff_i
      \\
      0 &~&  id_{\rr^{\Tau_\HH}}
    \end{array}
  \right].
\end{equation}
For $Y\in\ran{\sysmat}$, a solution to $\sysmat X=Y$ is 
provided by $X = U^{-1} \psinv{L} Y$.
\end{proposition}

\begin{remark}[About the time discretisation]
  \label{rem:time-disc}
  Choosing another time discretisation scheme will basically imply two changes: the computation of the right hand side (Step 1 in the resolution algorithm above) and the definition of $K$. 
  In general the global structure of the system matrix $\sysmat$ (which is symmetric positive semi-definite) as well as the positivity of $K$ will not be affected by considering different time discretisation: this is for instance the case for the Crank-Nicolson scheme or for operator splitting schemes (Strang formula e.g.).
\end{remark}

\begin{proof}[Proof of proposition \ref{prop:wp}]
For $X=\tsp{(U,V)}\in\rr^\Tau\times \rr^{\Tau_\HH}$, we have:
\begin{displaymath}
  \tsp{X} \sysmat X = 
  (\stiff_1 U,U)_\Tau
  +2(\stiff_i\rst U,V)_{\Tau_\HH}
  +(\stiff_i V,V)_{\Tau_\HH}
  +\gamma(\mass_\HH V,V)_{\Tau_\HH}
\end{displaymath}

We consider $\stiff_e$ and $\Gv$ defined in Rem. \ref{rem:K}.
Since $\Gun-\Gv$ is equal to 0 on $\TT$ and to $\Gi$ on $\HH$, $\stiff_1-\stiff_e$ is positive semi-definite.

Equation \eqref{eq:h4} says that $\left(\tsp{\rst}V\right)_{\vert
  \HH}=\rst\tsp{\rst}V=V$. Together with $\Gun-\Gv=0$ outside $\HH$ one gets:
\begin{align*}
  (\stiff_i V,V)_{\Tau_\HH} &=\int_\HH (\Gun-\Gv) \nabla V\cdot V ~dx
  \\
  &= \int_\oo (\Gun-\Gv)\nabla \tsp{\rst}
  V\cdot\nabla \tsp{\rst} V dx = 
  \left((\stiff_1-\stiff_e) \tsp{\rst} V,\tsp{\rst} V\right)_{\Tau}
  \\
  (\stiff_i\rst U,V)_{\Tau_\HH} &
  =\int_\HH (\Gun-\Gv)\nabla \rst U\cdot \nabla V dx
  \\
  &= \int_\oo (\Gun-\Gv)\nabla U\cdot\nabla 
  \tsp{\rst} V dx = \left((\stiff_1-\stiff_e) U,\tsp{\rst} V\right)_{\Tau}.
\end{align*}
From these two equalities we deduce that:
\begin{displaymath}
  \tsp{X} \sysmat X = 
  (\stiff_e U,U)_\Tau
  +
  \left((\stiff_1-\stiff_e) (U+\tsp{\rst} V), (U+\tsp{\rst} V) \right)_{\Tau}
  +\gamma(\mass_\HH V,V)_{\Tau_\HH} 
\end{displaymath}
so ensuring that $\sysmat$ is positive semi-definite. Assuming that $\sysmat X = 0$ implies
that all the terms on the right of the last equality are equal to zero. The mass
matrix being definite this means $V=0$ and so $\stiff_1 U =0$. Thus
 $U\in\ker{\stiff_1}=\ind_\oo \rr$ and we then have $\ker{\sysmat}=\ind_\oo\rr\times\{0\}$. 

Let
$X=\tsp{[U,V]}$ be a solution to  $\sysmat X=Y$
for $Y\in\ran{\sysmat}$. A simple computation shows that 
$Z=\tsp{[U-\alpha \ind_\oo,V]}$ is the unique solution to
\eqref{eq:prop-syslin} iff 
$
\alpha = (\mass U, \ind_\oo)_\Tau / (\mass \ind_\oo,\ind_\oo)_\Tau
$, so ending the proof.
\end{proof}

\begin{proof}[Proof of proposition \ref{prop:LU}]
  We have:
  \begin{displaymath}
    LU =     \left[
      \begin{array}{ccc}
        \stiff_1 &~& \proj_\oo \tsp{\rst}\stiff_i
        \\
        \stiff_i\rst &~& \gamma \mass_\HH + \stiff_i
      \end{array}
    \right],
  \end{displaymath}
  and so $LU=\sysmat$ \textit{iif} $ \proj_\oo \tsp{\rst}\stiff_i=
  \tsp{\rst}\stiff_i$. This last equality holds since for all $V\in\rr^{\Tau_\HH}$,
  \begin{displaymath}
    \left(\tsp{\rst}\stiff_i V, \ind_\oo\right)_\oo= 
    \left(\stiff_i V, \rst\ind_\oo\right)_\HH=
    \left(\stiff_i V, \ind_\HH\right)_\HH=0,
  \end{displaymath}
and so $\ran{\tsp{\rst}\stiff_i}\subset \ind_\oo^\perp$.
 
  The symmetry of $K$ is obvious. 
  Let us prove it is positive definite.
  \\
  We decompose $K = \gamma\mass_\HH + K_0$ so with 
  $K_0:=\stiff_i - \stiff_i\rst\psinv{\stiff_1} \tsp{\rst}\stiff_i$.
  We will prove that $K_0$ (which is symmetric) is positive semi-definite. This
  implies the positivity of $K$ since $\gamma\mass_\HH$ is
  positive definite. Precisely: $K_0$ clearly vanishes on $\ind_\HH
  \rr$. Then  $\ind_\HH^\perp$ is stable by $K_0$.
  Let us prove that $K_0$ is positive definite on $\ind_\HH^\perp$.
  
  We consider again
  $\stiff_e$ and $\Gv$ defined in Rem. \ref{rem:K}. 
  Let us first prove that:
  \begin{equation}
    \label{eq:pr1}
    K_0 = \rst\stiff_e \psinv{\stiff_1}\tsp{\rst}\stiff_i
  \end{equation}
  Firstly, we have: $\forall ~U_1.U_2\in\rr^\Tau$,
  \begin{displaymath}
    \int_\oo (\Gun-\Gv)\nabla U_1\cdot\nabla U_2 dx = 
    \int_\HH \Gi\nabla U_1\cdot\nabla U_2 dx,
  \end{displaymath}
  and so $\tsp{\rst} \stiff_i\rst = \stiff_1-\stiff_e$. 
  \\
  Secondly, multiplying $K_0$ by $\rst\tsp{\rst}=id_{\rr^{\Tau_\HH}}$ on
  the left gives:
  \begin{align*}
  K_0 = \rst \tsp{\rst} K_0 
  &=  
  \stiff_i -
  \rst \tsp{\rst}\stiff_i\rst\psinv{\stiff_1} \tsp{\rst}\stiff_i 
  \\
  &=
  \stiff_i -
  \rst (\stiff_1-\stiff_e)\psinv{\stiff_1} \tsp{\rst}\stiff_i 
  \\
  &= 
  \stiff_i -
  \rst (\proj_\oo-\stiff_e\psinv{\stiff_1} )\tsp{\rst}\stiff_i 
  \\
  &= 
  \rst\stiff_e  \psinv{\stiff_1}\tsp{\rst}\stiff_i
  +  \stiff_i -  \rst\proj_\oo\tsp{\rst}\stiff_i.
\end{align*}
One already showed in this proof that $ \proj_\oo \tsp{\rst}\stiff_i=
\tsp{\rst}\stiff_i$ ensuring that $\rst\proj_\oo\tsp{\rst}\stiff_i=
\stiff_i$. This gives us 
\eqref{eq:pr1}. 

Clearly $\psinv{\stiff_e}$ and
$\psinv{\stiff_i}$ are positive definite on $\ind_\oo^\perp$ and $\ind_\HH^\perp$
respectively. We moreover have $\tsp{\rst}(\ind_\HH^\perp)\subset
\ind_\oo^\perp$ since for all $V\in\ind_\HH^\perp$:
\begin{displaymath}
  \left(\tsp{\rst}V,\ind_\oo\right)_\oo =
  \left(V,\rst\ind_\oo\right)_\HH =
  \left(V,\ind_\HH\right)_\HH=0. 
\end{displaymath}
Then $\rst\psinv{\stiff_e}\tsp{\rst}$ is positive
definite on $\ind_\HH^\perp$. Let us define $A:=(\psinv{\stiff_i} +
\rst\psinv{\stiff_e}\tsp{\rst})$: $\ind_\HH^\perp$ is stable by $A$. $A$
is positive definite and so invertible on $\ind_\HH^\perp$.
We will end this proof by showing that $K_0=A^{-1}$ on $\ind_\HH^\perp$.
\begin{align*}
  K_0 A
  &= (\rst\stiff_e \psinv{\stiff_1}\tsp{\rst}\stiff_i)(\psinv{\stiff_i} + \rst\psinv{\stiff_e}\tsp{\rst})
  \\
  &=
  \rst \stiff_e \psinv{S_1} \tsp{\rst} \proj_\HH
  +
  \rst\stiff_e \psinv{\stiff_1}\tsp{\rst}\stiff_i
  \rst\psinv{\stiff_e}\tsp{\rst}
  \\
  &=
  \rst \stiff_e \psinv{S_1} \tsp{\rst} \proj_\HH
  +
  \rst\stiff_e
  \psinv{\stiff_1}(\stiff_1-\stiff_e)\psinv{\stiff_e}\tsp{\rst}
  \\
  &=
  \rst \stiff_e \psinv{S_1} \tsp{\rst} \proj_\HH
  +
  \rst\stiff_e
  (\proj_\oo\psinv{\stiff_e}-\psinv{\stiff_i}\proj_\oo)\tsp{\rst}  
  \\
  &=
  \rst \stiff_e \psinv{S_1} \tsp{\rst} \proj_\HH
  +
  \rst\stiff_e
  (\psinv{\stiff_e}-\psinv{\stiff_i})\tsp{\rst}  
  \\
  &=
  \rst\proj_\oo\tsp{\rst} + \rst \stiff_e \psinv{S_1} \tsp{\rst}(
  \proj_\HH - id_{\rr^{\Tau_\HH}}).
\end{align*}
Clearly, 
$\proj_\HH - id_{\rr^{\Tau_\HH}}$ vanishes on $\ind_\HH^\perp$. Moreover,
since $\tsp{\rst}(\ind_\HH^\perp)\subset
\ind_\oo^\perp$, $\rst\proj_\oo\tsp{\rst}$ is the identity on
$\ind_\HH^\perp$. Thus $K_0 A V = V$ for all $V\in\ind_\HH^\perp$. 
\end{proof}

\section{Preconditioning}
\label{sec:prec}

 The previously studied algebraic  properties of the system matrix $\sysmat$  naturally suggest a block-$LU$ designed preconditioner for $\sysmat$, here defined in Sec. \ref{sec:prec1}.
This general algebraic setting is the first key ingredient towards the preconditioning of the bidomain model. 
\\
The second key ingredient is a heuristic approximation of the matrix $K$, presented in Sec. \ref{sec:prec2}. 
\\
The last layer to practically implement the subsequent preconditioning indeed is discussed in Sec. \ref{sec:prec3}.

\subsection{Preconditioner definition}
\label{sec:prec1}
The practical strategy to solve \eqref{eq:system} will be to use an iterative solver for the left preconditioned system:
\begin{displaymath}
  \prec_\sysmat^{-1} \sysmat X = \prec_\sysmat^{-1} Y,
\end{displaymath}
for a global preconditioner $\prec_\sysmat$ defined as follows.
\begin{definition}
  \label{def:prec}
  Let us consider $\prec_1$ a preconditioner for $\stiff_1$ and $\prec_K$ a preconditioner for $K$. 
  We define a global preconditioner $\prec_\sysmat$ for $\sysmat$ as:
  \begin{align}
    \label{eq:prec-def}
    \prec_\sysmat
    \ = L_\prec U_\prec~,\quad 
    L_\prec:=\left[
      \begin{array}{ccc}
        \prec_1 &~&  0
        \\
        \stiff_i \rst &~& 
        \prec_K
      \end{array}
    \right] ,
    \quad 
    U_\prec:=\left[
      \begin{array}{ccc}
        id_{\rr^{\Tau}}&~&  \prec_1^{-1}\tsp{\rst}\stiff_i
        \\
        0 &~&  id_{\rr^{\Tau_\HH}}
      \end{array}
    \right].
  \end{align}
  The inversion of $P_\sysmat$ is achieved as follows. 
  The solution $X$ to $\prec_\sysmat X = Y$ is given by $X = U_\prec^{-1}L_\prec^{-1} Y$ with:
  \begin{equation}
    \label{eq:prec2}
    L_\prec^{-1}:=     
    \left[
      \begin{array}{ccc}
        \prec_1^{-1} &~&  0
        \\
        -\prec_K^{-1}\stiff_i\rst\prec_1^{-1} &~& 
        \prec_K^{-1}
      \end{array}
    \right] 
    ~,\quad 
    U_\prec^{-1}:=     
    \left[
      \begin{array}{ccc}
        id_{\rr^{\Tau}}&~&  -\prec_1^{-1}\tsp{\rst} \stiff_i
        \\
        0 &~& id_{\rr^{\Tau_\HH}}
      \end{array}
    \right]
    . 
  \end{equation}
\end{definition}
Neglecting the vector additions, the operational cost to compute $X=\sysmat Y$ is:
\begin{itemize}
\item[-] 2 multiplications by $\stiff_i$ 
\item[-] 1 multiplication by $\stiff_1$
\item[-] 1 multiplication by $\mass_\HH$,
\end{itemize}
whereas the operational cost to compute $X=P_\sysmat^{-1} Y$ is:
\begin{itemize}
\item[-] 2 inversions of $\prec_1$,
\item[-] 1 inversion  of $\prec_K$,
\item[-] 2 multiplications by $\stiff_i$,
\end{itemize}
The symmetry and positivity properties of $\sysmat$ allow to resort to a Preconditioned Conjugate Gradient (PCG) algorithm to solve \eqref{eq:system}.
The cost for this iterative solver (again neglecting scalar products and vector additions)
is for each step: one multiplication by $\sysmat$ and one inversion of $\prec_\sysmat^{-1}X=Y$. 

\subsection{Heuristic approximation of $K$}
\label{sec:prec2}
The hard task for the definition of $P_\sysmat$ in \eqref{eq:prec-def} is the definition of $P_K$.
As developed in Rem. \ref{rem:K}, $K$ has a complex structure:
\begin{displaymath}
  K = \gamma \mass_\HH + K_0,
\end{displaymath}
where $K_0$ is a non-sparse matrix obtained by making the harmonic mean between $\stiff_i$ and $\stiff_e$.
Since $K$ is a full matrix, it will never be computed and the alternative strategy to define $P_K$ is to derive an approximation of $K$ displaying a sparse pattern.

Let us consider the tensor $\G_{m}$:
\begin{equation*}
  \G_m (x) := (\G_e^{-1}(x) + \G_i^{-1}(x))^{-1}~,\quad x\in\HH,
\end{equation*}
which is the harmonic mean between $\G_i$ and $\G_e$. We introduce the stiffness matrix $\stiff_m$ associated to $\G_m$ acting on $\rr^{\Tau_\HH}$. We make the following approximation:
\begin{displaymath}
  K \simeq K_m := \gamma \mass_\HH + \stiff_m.
\end{displaymath}
This approximation is referred to as the \textit{monodomain model approximation} \cite{colli-taccardi-05}.
\\
The matrix $K_m$ has a simple structure. It is  the discretisation matrix of a parabolic equation. It is moreover symmetric, positive definite and sparse (with the same pattern as $S_i$). 

\subsection{Practical implementation of $P_1$ and $P_K$}
\label{sec:prec3}
The two preconditioners $P_1$ and $P_K$ will be built from the matrices $\stiff_1$ and $K_m$ respectively.
These matrices (sparse, symmetric positive semi-definite) have classical structures arising from the discretisation of elliptic and parabolic problems respectively.
A wide literature has been devoted to the preconditioning of such matrices: among classical choices we not comprehensively quote incomplete decomposition methods (incomplete $LU$ or incomplete Cholesky, see e.g. \cite{saad}) multi-grid or multi-level methods, see \cite{hackbusch-book-85,ccg-96}.
Fixing one of these classical possible choices actually provide a fully defined implementation of the here presented bidomain model preconditioning.
\\
We insist on the versatility of this bidomain model preconditioning.
This versatility  relies on the freedom for the choice of $P_1$ and $P_K$. 
\begin{remark}[Parallelisation]
  At this stage, let us underline the consequences on parallelisation induced by this versatility  characteristic of the bidomain model preconditioning.
Once embedded into some iterative solver (e.g. CG or GMRes) the resolution of system \eqref{eq:system} preconditioned by $P_\sysmat$ only requires:
\begin{itemize}
\item[-~] matrix vector multiplications by $\sysmat$,
\item[-~] inversions of $P_\sysmat X =Y$: as detailed in Sec. \ref{sec:prec3} this operation consists in matrix vector multiplication and inversions of $P_1 X =Y$ and of $P_K X =Y$,
\item[-~] various remaining operations, such as scalar products..
\end{itemize}
Except the inversions of $P_1 X =Y$ and of $P_K X =Y$, all these operations have trivial parallelisation.
But since $P_1$ and $P_K$ are preconditioners for classical elliptic or parabolic discretised PDEs, classical parallel versions for $P_1$ and $P_K$ already are available. For instance a review of algebraic methods (such as parallel version of incomplete factorisations) is provided in 
\cite{benzi-2002,saad}. Another wide class of parallelisation strategies based on domain decomposition is analysed in \cite{quart-valli-dom-dec} and also described in \cite{saad}. For instance the multi-level additive Schwarz preconditioner, such as presented in \cite{pavarino-schwarz-2008} and applied to the bidomain model, also could be incorporated inside the here presented general preconditioning framework.
\\
For this reason, the here presented preconditioning strategy for the bidomain model naturally fits with the constraints of parallelism.
\end{remark}
%
Optimal complexity to solve a discretised elliptic problem $AX=Y$ is $O(n)$ with $n$ the system size: since $X\mapsto AX$ has $O(n)$ complexity one cannot hope better for $Y\mapsto A^{-1}Y$ ($A$ being sparse whereas $A^{-1}$ is full). Although this optimality can be reached for some particular problems (for instance in case $A$ is tri-diagonal), in practise the most efficient algorithms have \textit{almost linear complexity}: that is $O(n \log(n)^\alpha)$ with $\alpha$ a constant. 

Hierarchical matrices preconditioning strategy 
\cite{hackbusch-2001-Hmat,lars-hackbusch-2002,lars-hackbusch-2003,lars-leborne-2008}
provides such an almost linear complexity (among various possible choices such as multi-grid methods \cite{hackbusch-book-85}). 
This method will be used for the numerical results in Sec. \ref{sec:num} to precondition $S_1$ and $K_m$. 
This method proceeds in two steps.
Firstly compute an approximation of the considered matrix (here $S_1$ or $K_m$).
This approximation is built using hierarchical matrices arithmetic (basically including  block partition of the matrix and defining a blockwise approximation by low rank matrices), ensuring low storage cost.
This approximation accuracy is controlled by the parameter $\epsilon$: in matrix norm the error goes to 0 with $\epsilon$.
Secondly perform the exact decomposition (either $LU$ or Cholesky) of this approximation.
Hierarchical Cholesky decomposition has been used here to build $P_1$ and $P_K$. Taking advantage of the hierarchical arithmetic, both the construction, storage and inversion of the preconditioners are in $O(n\log(n)^\alpha)$, precisely with $\alpha=2$ (resp. 4) for the decomposition and $\alpha=1$ (resp. 2) for the storage/inversion in dimension 2 (resp. 3).
%
\\
The setting of the accuracy parameter $\epsilon$ strongly impacts the preconditioning efficiency.
Naturally the PCG convergence rate increases as $\epsilon$ goes to 0.
A convergence in one single PCG iteration is expected provided a small enough value for $\epsilon$.
Meanwhile the preconditioner inversion cost increases as $\epsilon\mapsto  0$: thus the highest PCG convergence rate may not correspond to the most efficient setting of the preconditioner. 
An optimal value for $\epsilon$ (not too small but not too large) has to be searched. PCG convergence rate for such optimal value are shown in Sec. \ref{sec:res} for which   3 PCG iterations typically have to be performed.

In practise the construction of  $P_1$ and $P_K$ was made using the H-Lib library from L. Grasedyck and S. B\"{o}rm\footnote{http://www.hlib.org/}. 
The sequential version of the code has been used: a parallel version also is  available.
\section{Numerical results}
\label{sec:num}
The efficiency of the preconditioner presented in Sec. \ref{sec:prec} is analysed in this section. 
The bidomain model has been implemented following Sec. \ref{sec:methods} and using the CVFE finite volume spatial discretisation (see e.g. \cite{cai_etal_1991}). For this spatial discretisation the degrees of freedom are located at the mesh vertices and the mass matrices are diagonal.
Two test cases are considered,
they are detailed in Sec. \ref{sec:ionic-model}.
For these two test cases a depolarisation potential wave is simulated. 
The spreading of depolarisation 
The cost for the inversion of the preconditioned system 
\eqref{eq:system} is measured during the spreading of the depolarisation wave, that numerically is by far the stiffest part of the simulation.
The dependence of this cost on the problem size is then analysed.
For this a series of meshes $\Tau_n$ is considered with an increasing number of vertices $\text{DOF}(n)$. 
We here aim to validate an almost linear dependence of the cost on $\text{DOF}(n)$.
\\
The cost has been measured in two ways. 
Firstly in terms of CPU time.
The averaged CPU time spent on the inversion of system \eqref{eq:system}  during the depolarisation sequence is denoted $\text{CPU}(n)$.
The logarithmic growth rate $r_n$ of $\text{CPU}(n)$ relatively to  $\text{DOF}(n)$ will be considered:
\begin{equation}
  \label{eq:growth-rate}
  r_n =
  \dfrac{\log(\text{CPU}(n)/\text{CPU}(n-1))}
  {\log(\text{DOF}(n)/\text{DOF}(n-1)}.
\end{equation}
The CPU time measurements however might  be perturbed by cache effects and memory-access differences for large-scale problems. 
To cope with this, the cost also is evaluated in terms of number of iterations.
The averaged number of iterations required by the PCG algorithm to invert \eqref{eq:system}  during the depolarisation sequence is denoted $\text{Iter}(n)$. Each step of the PCG algorithm requires one multiplication by $\sysmat$ and one inversion of $P_\sysmat$. These operations  are of linear and almost linear complexity with $\text{DOF}(n)$ respectively.
Thus  a constant or logarithmic behaviour is expected for $\text{Iter}(n)$ to validate an almost linear complexity of the preconditioning.
\\
Numerical results for the preconditioning complexity are presented and discussed in Sec. \ref{sec:res} and \ref{sec:conc} respectively.

\subsection{Test cases}
\label{sec:ionic-model}
\begin{table}
  \centering
  \begin{tabular}[h]{lclc}
    \hline
    & &
    \textbf{Values} &
    \textbf{Unit}
    \\
    \hline
    Cell membrane surface-to-volume ratio (2D) & \quad& $\am=1500$ &  [cm$^{-1}$]
    \\
    Cell membrane surface-to-volume ratio (3D) & \quad& $\am=500$ &  ''
    \\
    Membrane surface capacitance & & $\cm=1.0$  &  [$\mu$ F/cm$^2$]
    \\
    Longitudinal intra-cellular conductivity && $g_i^{l}= 1.741$ & [mS/cm]
    \\
    Transverse intra-cellular conductivity && $g_i^{t}= 0.1934$ &  ''
    \\
    Longitudinal extra-cellular conductivity && $g_e^{l}= 3.906$ &  ''
    \\
    Transverse extra-cellular conductivity && $g_e^{t}= 1.970$ &  ''
    \\
    Lung conductivity &&  $ 0.5$ & ''
    \\
    Blood conductivity (ventricular cavities) && $ 6.7$ &  ''
    \\
    Remaining tissues conductivity  &&   $ 2.2$ &  ''
    \\ \hline \\[3pt]
  \end{tabular}
  \caption{Model parameters}
  \label{tab:parameters}
\end{table}
For the two test cases, the reaction terms $\ion(\vm,\ww)$ and $g(\vm,\ww)$ in \eqref{eq:bid} have been set to the Luo and Rudy ionic model of  class II \cite{LR2a} designed for mammalian ventricular cells and for which the system of ODEs in \eqref{eq:bid} is of size 20 (i.e. $\ww\in\rr^{20}$). 
The model parameters $\am$, $\cm$ as well as the conductivities  are displayed in Tab.\ref{tab:parameters}: these values are physiological values taken from \cite{leguyader_01,buist-pullan-smith-03}.  

\vspace{5pt}
\textbf{2D test case.} 
\begin{figure}[!ht]
  \centering
  \begin{tabular}{ccc}
      \includegraphics[width=120pt]{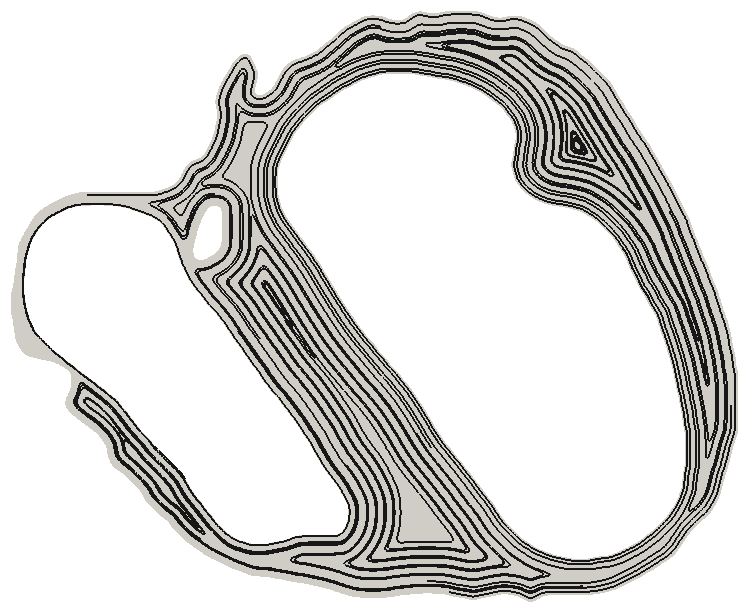}
      &
      \includegraphics[width=150pt]{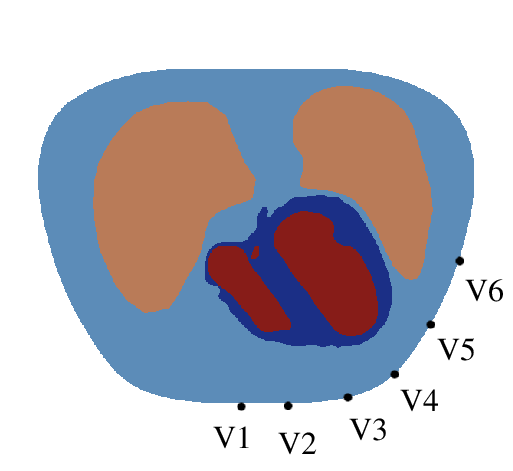}
      &
      \includegraphics[width=120pt]{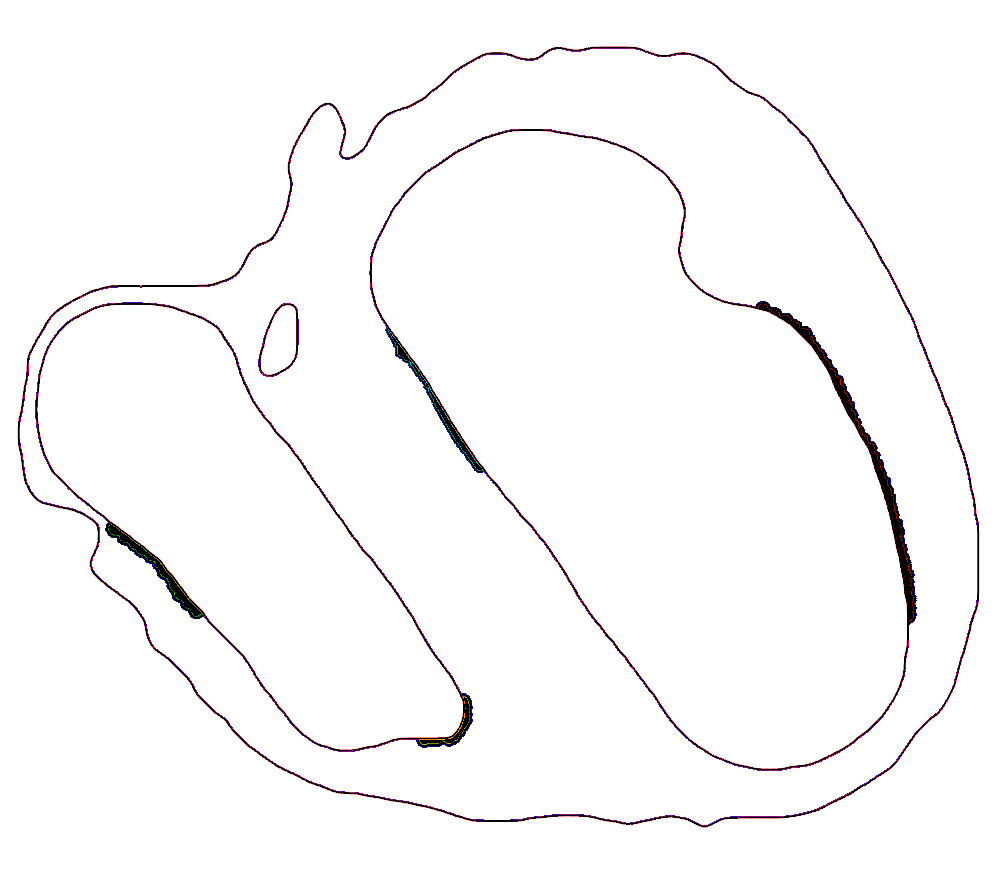}
  \end{tabular}
  \caption{2D test case description. Left: fibrous anisotropic structure of the two ventricles. 
    Middle: 2D geometry $\oo$ and its sub-domains. 
    body surface potential (ECG) are recorded at the vertices V1 to V6.
    Right: stimulation site locations.
  }
  \label{fig:test-case}
\end{figure}
The domain $\oo$ is an horizontal slice of a human thorax. This geometry has been obtained by segmentation of a medical image (CT-Scan, courtesy of the Ottawa Heart Institute) with resolution 0.5 $mm$. 
We refer to \cite{oli-these,oli-2009} for details on the segmentation procedure.
The segmented image is depicted in Fig. \ref{fig:test-case}. 
It includes 4 sub-domains: the two ventricles ($\HH$) and the torso ($\TT$) made of the ventricular cavities, the lungs and the remaining tissues. 
\\
Four meshes $(\Tau_n)_{n=1\dots 4}$ of $\oo$ will be considered: with
DOF(1)=143 053, DOF(2)=344 408, DOF(3)=684 112 and DOF(4)=1 257 312.
The associated time steps are $\Delt=$ 0.07, 0.05, 0.035 and 0.025 milli seconds (ms) respectively.
\\
The anisotropic structure of the two ventricles is displayed on Fig. \ref{fig:test-case}: bundles of fibres rotating around the ventricular cavities have been considered. Inside the torso $\TT$, heterogeneous conductivities have been considered for each sub-domains: the lungs, ventricular cavities and the remaining tissues conductivities are given in Tab. \ref{tab:parameters}.
\\
With these settings, a depolarisation potential wave is simulated. 
For this a stimulation current $\iapp(x,t)$ (see equation \eqref{eq:bid}) is applied during 1 ms at four locations (stimulation sites) on the ventricular cavities as depicted on Fig. \ref{fig:test-case};
the right ventricle being stimulated 5 ms later than the left one.
\\
The spreading of this potential wave across the myocardium is depicted on Fig. \ref{fig:2d-dep}. 
The transmembrane potential $\vm$ in the heart is depicted 15, 30 and 45 ms after stimulation on the left. 
Without entering the details: the region in blue is at rest potential ($\vm\simeq-90$ mV) whereas the region in red is excited ($\vm\simeq 50$ mV). 
Downward: the excitation wave starts at the stimulation site location and then spreads throughout the cardiac tissue.
The activation time $\phi(x)$ is computed pointwise as the time $t=\phi(x)$ so that $\vm(\phi(x),x)=-20$ mV (the time instant when the depolarisation wave reaches the point $x$). Activation time are depicted on Fig. \ref{fig:2d-ecg}.
\\
The modifications on the extra-cellular (and extra-cardiac) potential $\uu$ on $\oo$ (heart and torso) induced by the transmembrane depolarisation wave spreading also is depicted on Fig. \ref{fig:2d-dep}.
The body surface potential (ECG) is recorded at 6 points on $\partial\oo$, their location is depicted on Fig. \ref{fig:test-case} (points V1 to V6). These potentials $(\uu(t,Vi))_{i=1\dots 6}$ are recorded at each time step along a complete cardiac cycle (including depolarisation and repolarisation). Results are depicted on Fig. \ref{fig:2d-ecg} on the right for the two electrodes V2 and V6.

\begin{figure}[!ht]
  \centering
  \begin{tabular}{cc}
      \includegraphics[width=200pt]{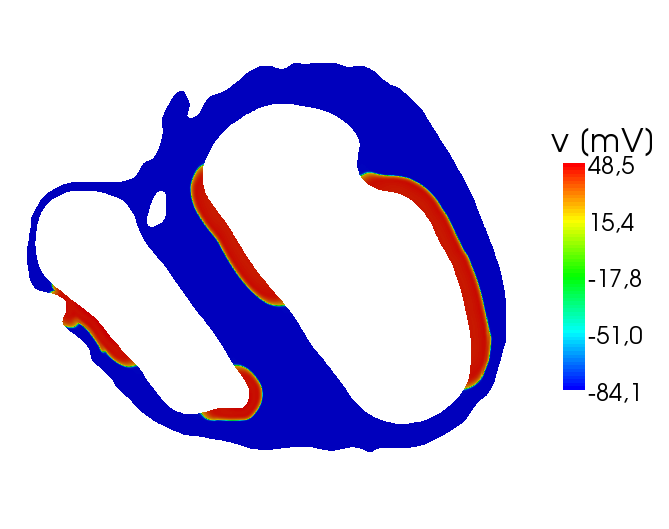}
      &
      \includegraphics[width=200pt]{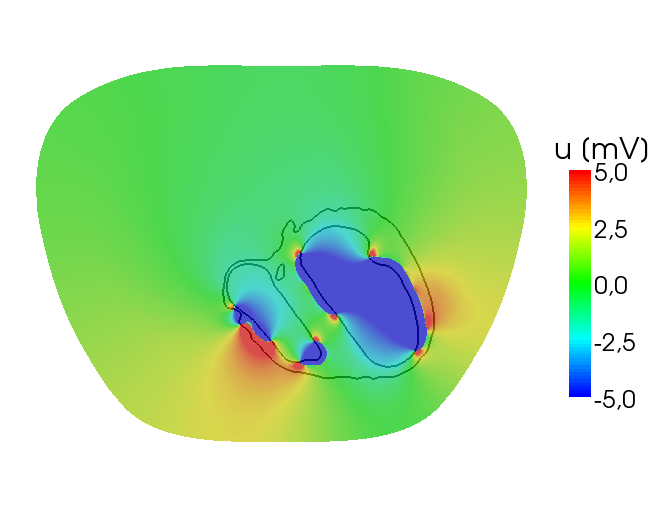}
      \\
      \includegraphics[width=200pt]{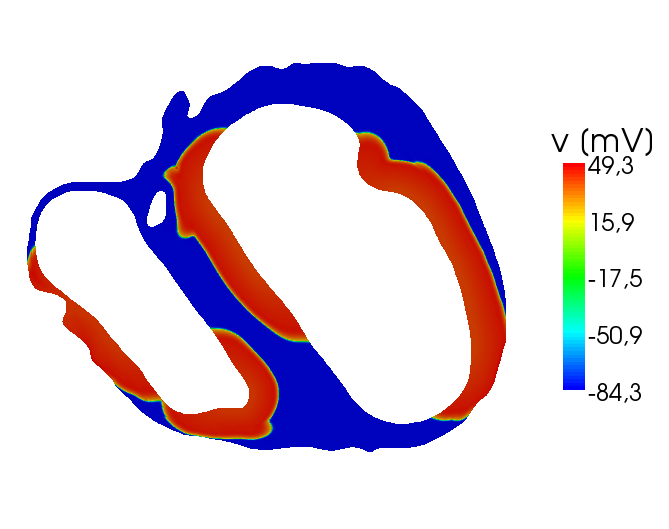}
      &
      \includegraphics[width=200pt]{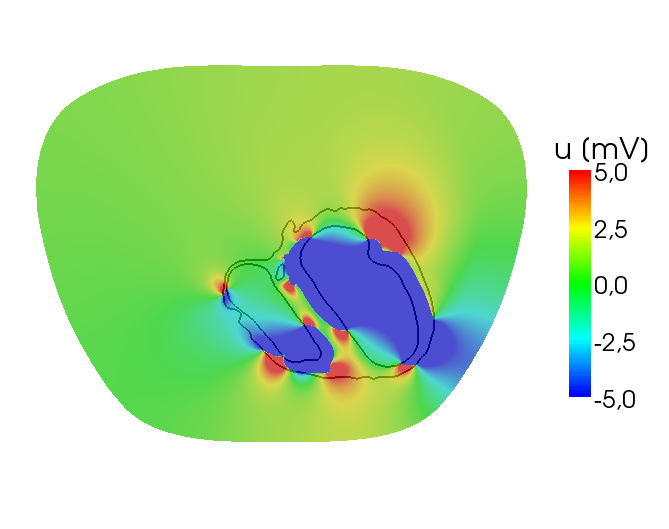}
      \\
      \includegraphics[width=200pt]{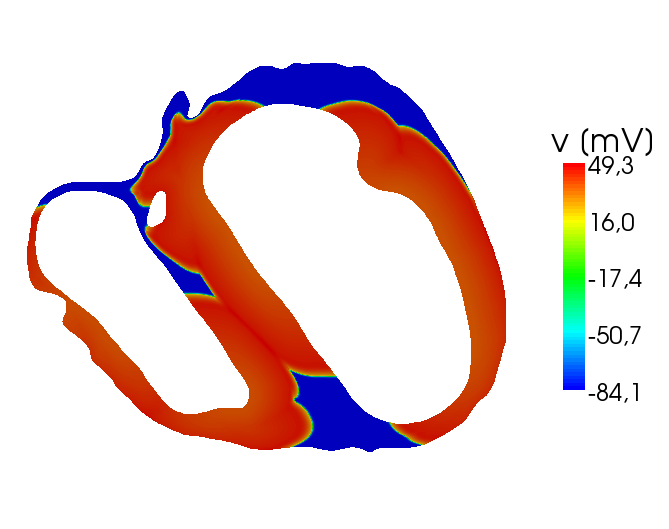}
      &
      \includegraphics[width=200pt]{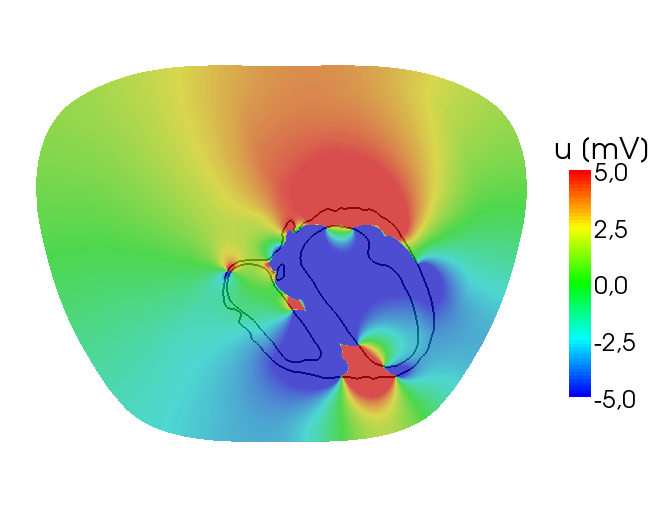}
  \end{tabular}
  \caption{2D simulation. Left: depolarisation sequence of the heart, the transmembrane potential $\vm$ is represented 15, 30 and 45 ms after stimulation. Right: associated potential $\uu$ in the heart and in the extra cardiac region.
  }
  \label{fig:2d-dep}
\end{figure}
\begin{figure}[!ht]
  \centering

  \begin{tabular}{cc}
    \begin{tabular}{c}
      \includegraphics[width=200pt]{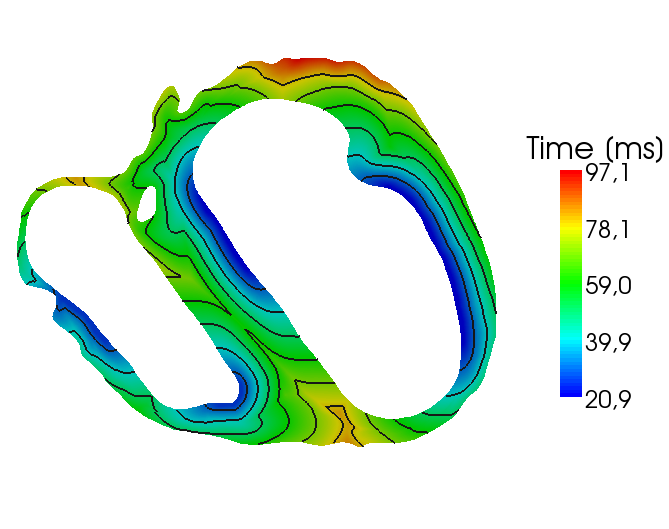}
    \end{tabular}
    &
    \begin{tabular}{c}
       \includegraphics[width=220pt]{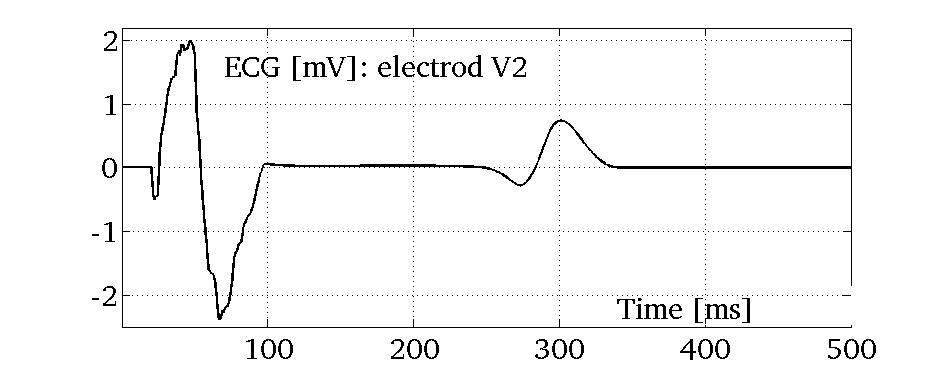}
       \\
       \includegraphics[width=220pt]{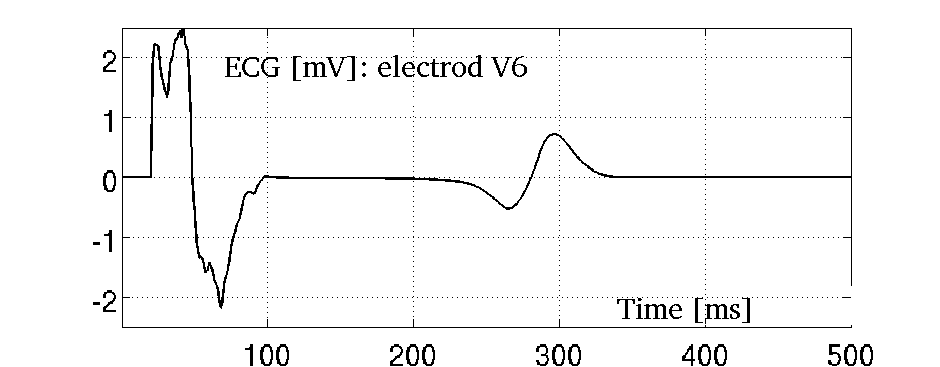}      
    \end{tabular}
  \end{tabular}
  \caption{2D simulation. Left: activation time in the heart, isolines in black are separated by 10 ms.
    Right: ECG recordings, the extra-cardiac potential is recorded on the torso surface at two points located at electrodes V2 (above) and V6 (below), see figure \ref{fig:test-case} for the electrode location.
  }
  \label{fig:2d-ecg}
\end{figure}

\vspace{5pt}
\textbf{3D test case.} 
We here consider a small slab of tissue: 
a cubic domain with one centimetre width ($\oo=[0,1]^3$).
A series of 5 meshes $(\Tau_n)_{n=1\dots 5}$ has been considered, from 500 to 1 250 000 vertices (see Tab. \ref{tab:res-cost2} for exact figures).
The mesh size being divided by 2 from $\Tau_n$ to $\Tau_{n+1}$, the time stepping $\Delt$ also is divided by 2 and ranges from 0.2 to 0.0125 ms from the coarsest to the finest mesh.
The heart is here considered as isolated: no torso $\TT$ is involved as described in Sec. \ref{subsec:isol-case}. 
The cardiac tissue anisotropy is set to be of orthotropic type, as defined in \cite{colli-taccardi-05}. 
Muscular fibres are horizontal and independent of $x$ and $y$. The fibre directions linearly rotate from $+\pi/4$ to $-\pi/4$ as $z$ goes from 0 to 1. 
Orthotropic anisotropy represents the physiologically observed rotation of the cardiac fibres from $+\pi/4$ to $-\pi/4$ from the endo-cardium to the epi-cardium.
\\
A depolarisation potential wave is simulated by applying a stimulation current at the centre of the domain during 1 ms.
\\
The spreading of transmembrane depolarisation wave is depicted on Fig. \ref{fig:3d-dep}. 
Activation time are here represented for three slices of the domain $\Omega=[0,1]^3$: $z=0$, $z=0.5$ and $z=1$. Each slice  corresponds to the endo-cardium, middle wall and epi-cardium respectively.
The fibre angle with $\mathbf{e}_x$ is clearly visible on each slice:
$+\pi/4$ for $z=0$ (left), 0 for $z=0.5$ (middle) and $-\pi/4$ for $z=1$ (right).
\begin{figure}[!ht]
  \centering
  \begin{tabular}{ccc}
      \includegraphics[width=150pt]{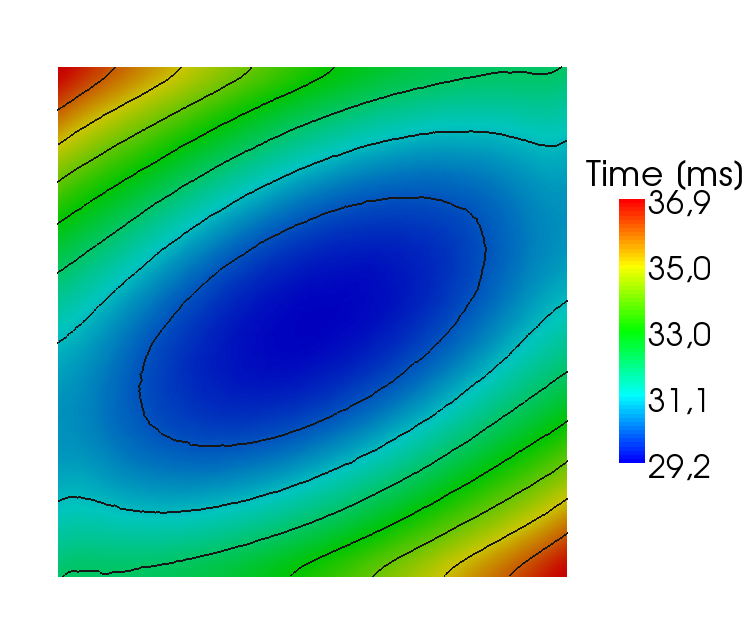}
      &
      \includegraphics[width=150pt]{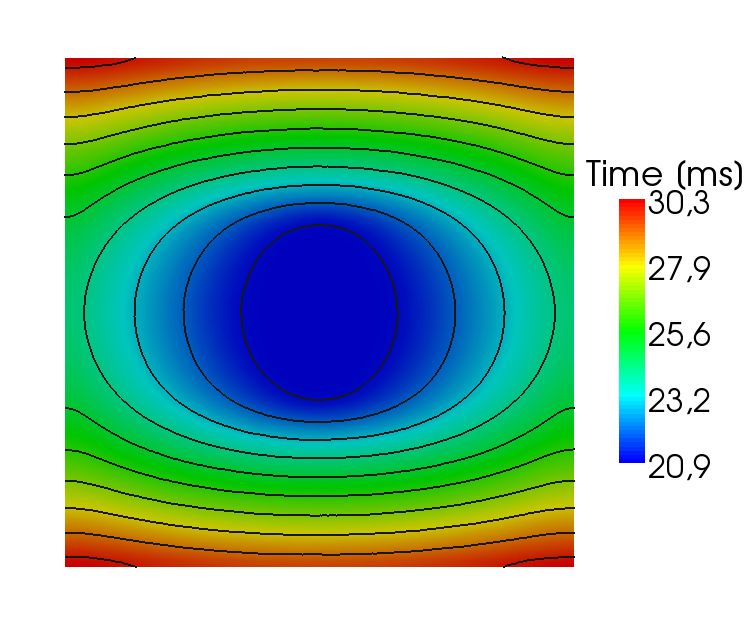}
      &
      \includegraphics[width=150pt]{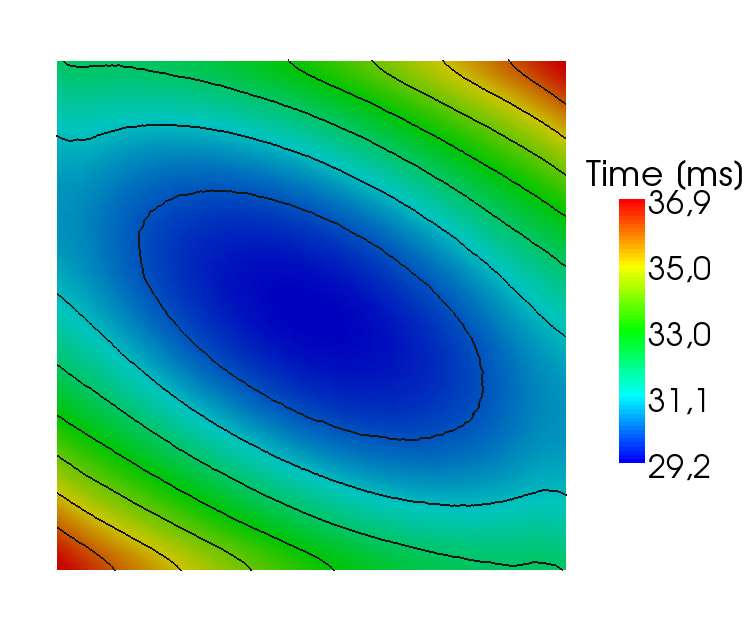}
  \end{tabular}
  \caption{3D simulation. Activation times for three slices of the domain $\Omega=[0,1]^3$: $z=0$, $z=0.5$ and $z=1$ from left to right. 
    Isolines (in black) are separated by 1 ms.
  }
  \label{fig:3d-dep}
\end{figure}

\subsection{Results}
\label{sec:res} 
All figures and tables reported here have been obtained fixing a tolerance of $10^{-6}$ for the system \eqref{eq:system} inversion; the residual being defined as $\Vert \sysmat X -Y\Vert /\Vert  Y\Vert $ in Euclidian vector norm.
The hierarchical Cholesky decompositions for $P_1$ and $P_K$ have been built for various values of the accuracy parameter $\epsilon$
introduced in Sec. \ref{sec:prec3}.
All computations were ran on a clustered platform with processor cores of type AMD Opteron, 2.3 GHz.

\begin{table}[!ht]
  \flushleft
  \hspace*{-35pt}
  \begin{tabular}{cc}
    \begin{tabular}{|c|c||c|c|c|}
      \hline     
      $n$ & DOF($n$) & 
      \multicolumn{3}{|c|}{Iter($n$)}
      \\ \cline{3-5}
      & &
      $\epsilon=10^{-2}$ & 
      $\epsilon=10^{-3}$ & 
      $\epsilon=10^{-4}$
      \\[1pt]
      \hline
      1 & 143 053    & 3.19  &3.00&3.00
      \\
      2 & 344 408    & 3.82  &3.00&3.00
      \\
      3 & 684 112    & 4.00  &3.00&3.00
      \\
      4 & 1 257 312  & 4.54  &3.00&3.00
      \\[1pt]
      \hline
      \multicolumn{5}{c}{}
    \end{tabular}
    &
    \begin{tabular}{|c|c||c|c|c|}
      \hline
      
      $n$ & DOF($n$) & 
      \multicolumn{3}{|c|}{Iter($n$)}
      \\ \cline{3-5}
      & &
      $\epsilon=10^{-1}$ & 
      $\epsilon=10^{-2}$ & 
      $\epsilon=10^{-3}$
      \\[1pt]
      \hline
      1 & 497       &  2.40  &  2.00  & 2.00 
      \\
      2 & 3 220     &  4.03  &  2.79  & 2.76
      \\
      3 & 22 256    &  5.14  &  3.00  & 3.00
      \\
      4 & 162 981   &  7.43  &  3.24  & 3.00
      \\
      5 & 1 253 910 &  11.20 &  3.96  & 2.00
      \\[1pt]
      \hline
    \end{tabular}
    \\
    (a) 2D case &    (b) 3D case 
    \\ $~$
  \end{tabular}
  \caption{Average number of iterations for one system inversion.}
  \label{tab:res-cost2}
\end{table}

\vspace{5pt}
\textbf{Number of iterations.}
We first investigate the cost for system \eqref{eq:system} during the depolarisation sequence in terms of number of iterations Iter($n$) for the PCG algorithm.
As already developed in this section preamble, the global cost  theoretically is in $O(\text{Iter}(n)\text{DOF}(n)\log(\text{DOF}(n))^\alpha)$.

The numerical results are reported in Tab. \ref{fig:iter-fig}.
In dimension 2, for $\epsilon=10^{-2}$ Iter($n$) globally is
 multiplied by 1.18 between the coarsest and the finest meshes when meanwhile the problem size is multiplied by almost 9. For $\epsilon\le 10^{-3}$ Iter($n$) remains constant. In dimension 3 Iter($n$) increases very slowly: for $\epsilon=10^{-2}$ (resp. $10^{-1}$) it is multiplied by 2 (resp. 4.66) when the problem size is multiplied by more than 2 500; for $\epsilon=10^{-2}$ it even decreases.
\begin{figure}[!ht]
  \centering
  \begin{tabular}{ccc}
    \includegraphics[width=220pt]{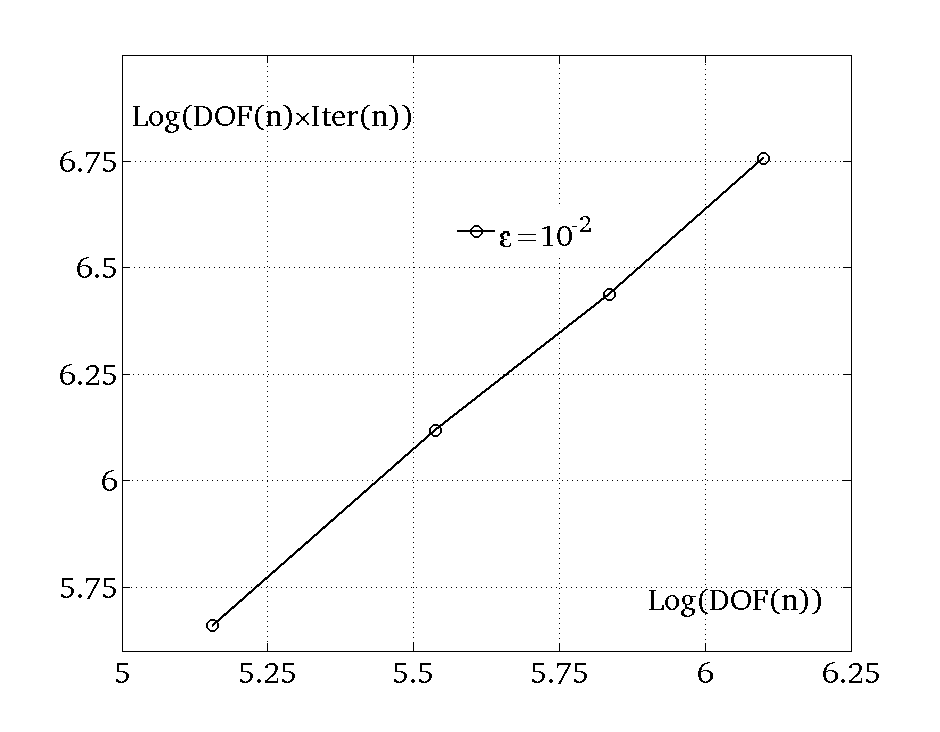} 
    &
    \includegraphics[width=220pt]{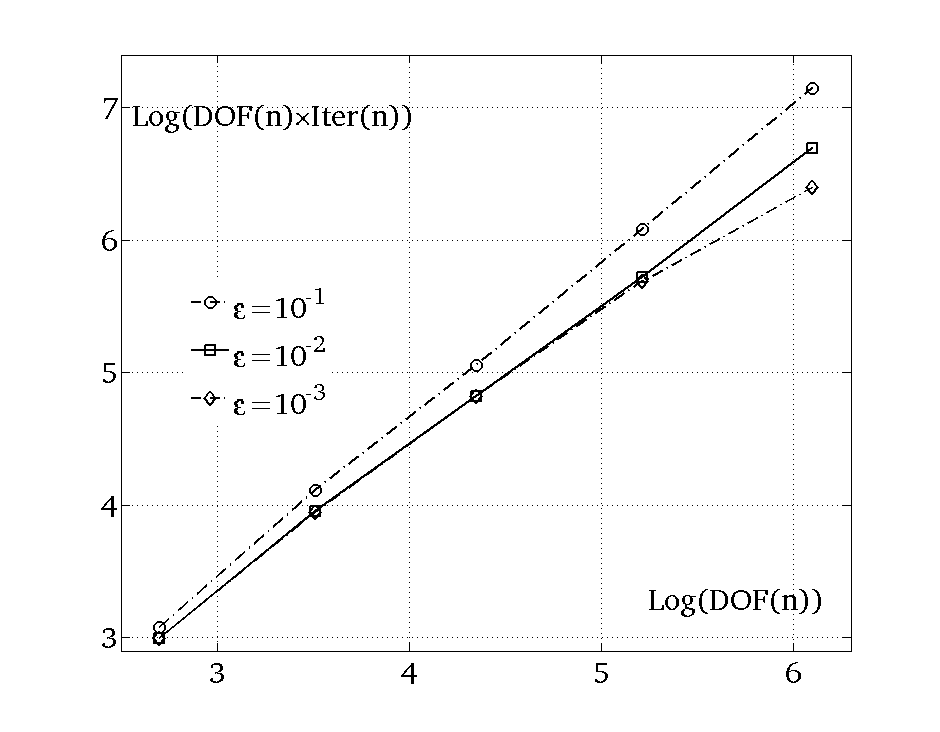} 
    \\
    (a) 2D case &    (b) 3D case 
  \end{tabular}
  \caption{
Plot of DOF($n$)$\times$Iter($n$) as a function of DOF($n$) in (decimal) Log/Log Scale. Left: 2D case for $\epsilon=10^{-2}$. Right: 3D case for the three values of $\epsilon=10^{-2},~10^{-3}$ and 10$^{-3}$.
}
  \label{fig:iter-fig}
\end{figure}

The very slow variation of Iter($n$) with DOF($n$) (when it is not constant) appears in good agreement with a $O(\log(\text{DOF}(n))^\beta)$ assumption ensuring almost  linear complexity of the preconditioning global cost.
It is unfortunately not possible to numerically estimate $\beta$ from these results since $\log(\log(\text{DOF}(n)))$ has a too small range of variation.
To have a deeper insight on the behaviour of Iter($n$) when it does not remain constant we instead consider the cost indicator $\text{DOF}(n)\times \text{Iter}(n)$. An almost linear behaviour of this indicator is expected. It has been represented as a function of $\text{DOF}(n)$ in decimal logarithmic scale on Fig. \ref{fig:iter-fig}.
In dimension 2 the curve has a global estimated slope of 1.15 using a linear least square best approximation. In dimension 3 the slopes  have been estimated to 1.19, 1.07 and 1.0 for $\epsilon=10^{-1}$, $10^{-2}$ and $10^{-3}$ respectively.
Again, these results are in good agreement with the almost linear complexity assumption on the preconditioning.

\begin{table}[!ht]
  \centering
  \begin{tabular}{cc}
    \begin{tabular}{|c|c||c|c|c|}
      \hline
      $n$ &
      DOF$(n)$ & 
      \multicolumn{3}{|c|}{CPU$(n)$}
      \\ \cline{3-5}
      & &
      $\epsilon=10^{-2}$ & 
      $\epsilon=10^{-3}$ & 
      $\epsilon=10^{-4}$
      \\[1pt]
      \hline
      1& 143 053    & 1.73   & 1.57   & 1.78 
      \\
      2& 344 408    & 6.32   & 4.34   & 4.42 
      \\
      3 &684 112    & 10.49  & 8.75   & 8.39
      \\
      4 &1 257 312  & 23.96  & 17.04  & 13.46 
      \\[1pt]
      \hline
      \multicolumn{5}{c}{}
    \end{tabular}
    &
    \begin{tabular}{|c||c|c|c|}
      \hline
      
      $n$ & 
      \multicolumn{3}{|c|}{$r_n$}
      \\ \cline{2-4}
      & 
      $\epsilon=10^{-2}$ & 
      $\epsilon=10^{-3}$ & 
      $\epsilon=10^{-4}$
      \\[1pt]
      \hline
      2   &1.47  & 1.16& 1.04
      \\
      3   &0.74  & 1.02& 0.93
      \\
      4   &1.36  & 1.09& 0.78
      \\[1pt]
      \hline 
      \multicolumn{4}{c}{}
      \\
      \multicolumn{4}{c}{}
    \end{tabular}
  \end{tabular}
  \caption{CPU Time, 2D case. Left: averaged CPU time in seconds for one system inversion. 
    Right: logarithmic growth of CPU($n$)with respect to DOF($n$).
  }
  \label{tab:cpu-2d}
\end{table}
\begin{table}[!ht]
  \centering

  \begin{tabular}{ccc}

    \begin{tabular}{|c|c||c|c|c|}
      \hline
      
      $n$ & DOF$(n)$ & 
      \multicolumn{3}{|c|}{CPU($n$)}
      \\ \cline{3-5}
      & &
      $\epsilon=10^{-1}$ & 
      $\epsilon=10^{-2}$ & 
      $\epsilon=10^{-3}$
      \\[1pt]
      \hline
      1& 497       & 2.0 $10^{-3}$ & 1.7 $10^{-3}$ & 1.8 $10^{-3}$
      \\
      2& 3 220     & 5.1 $10^{-2}$ & 4.1 $10^{-2}$ & 4.2 $10^{-2}$
      \\
      3& 22 256    & 6.9 $10^{-1}$ & 4.4 $10^{-1}$ & 4.9 $10^{-1}$
      \\
      4& 162 981   & 8.6           & 4.6           & 5.5
      \\
      5& 1 253 910 & 102.96        & 59.8          & 32.2
      \\[1pt]
      \hline
      \multicolumn{5}{c}{}
    \end{tabular}
    &
    \begin{tabular}{|c||c|c|c|}
      \hline
      $n$ & 
      \multicolumn{3}{|c|}{$r_n$}
      \\ \cline{2-4}
      & 
      $\epsilon=10^{-1}$ & 
      $\epsilon=10^{-2}$ & 
      $\epsilon=10^{-3}$
      \\[1pt]
      \hline
      2   &  1.75  &  1.70  &  1.70
      \\
      3   &  1.34  &  1.22  &  1.27
      \\
      4   &  1.27  &  1.19  &  1.21
      \\
      5   &  1.22  &  1.25  &  0.86
      \\[1pt]
      \hline
      \multicolumn{4}{c}{}
      \\
      \multicolumn{4}{c}{}
    \end{tabular}
  \end{tabular}
  \caption{CPU Time, 3D case. Left: averaged CPU time in seconds for one system inversion. 
    Right: logarithmic growth of CPU($n$)with respect to DOF($n$).}
  \label{tab:cpu-3d}
\end{table}

\vspace{5pt}
\textbf{CPU time consumption.}
The cost CPU($n$) is reported in Tab. \ref{tab:cpu-2d} (resp. Tab. \ref{tab:cpu-3d}) in dimension 2 (resp. 3) together with the logarithmic growth rate $r_m$ of CPU($n$) with respect to DOF($n$) defined in \eqref{eq:growth-rate}. As for the iteration number, the behaviour of CPU($n$) is clearer for the smallest values of $\epsilon$. For $\epsilon\le 10^{-3}$ (resp. $\epsilon\le 10^{-2}$) in dimension 2 (resp. 3), $r_n$ decreases with $n$ and goes to 1 or even below 1.

The data in Tabs. \ref{tab:cpu-2d} and \ref{tab:cpu-3d} have been plotted on Fig. \ref{fig:res-cout}. The curve slopes have been estimated using a least square best linear approximation. 
In dimension 2 the slopes  are of 1.17, 1.09 and 0.94 for $\epsilon=10^{-2}$, $10^{-3}$ and  $10^{-4}$ respectively.
In dimension 3 they  are of 1.27, 1.21 and 1.12
for $\epsilon=10^{-1}$, $10^{-2}$ and  $10^{-3}$ respectively (and neglecting the first data point).
\\
Firstly, since $r_n$ roughly decreases (starting with rates higher than 1.7 in dimension 3), these computed slopes indeed are upper-bounds on the complexity.
Secondly CPU time is not a fully reliable cost measurement: because of cache effects memory-access differences for large-scale problems and because of the cluster load.
For these two reasons we conclude that these CPU data are in good agreement with an almost linear complexity of the preconditioned system inversion, confirming the study of Iter($n$). 
\begin{figure}[!ht]
  \centering
  \begin{tabular}{ccc}
    \includegraphics[width=220pt]{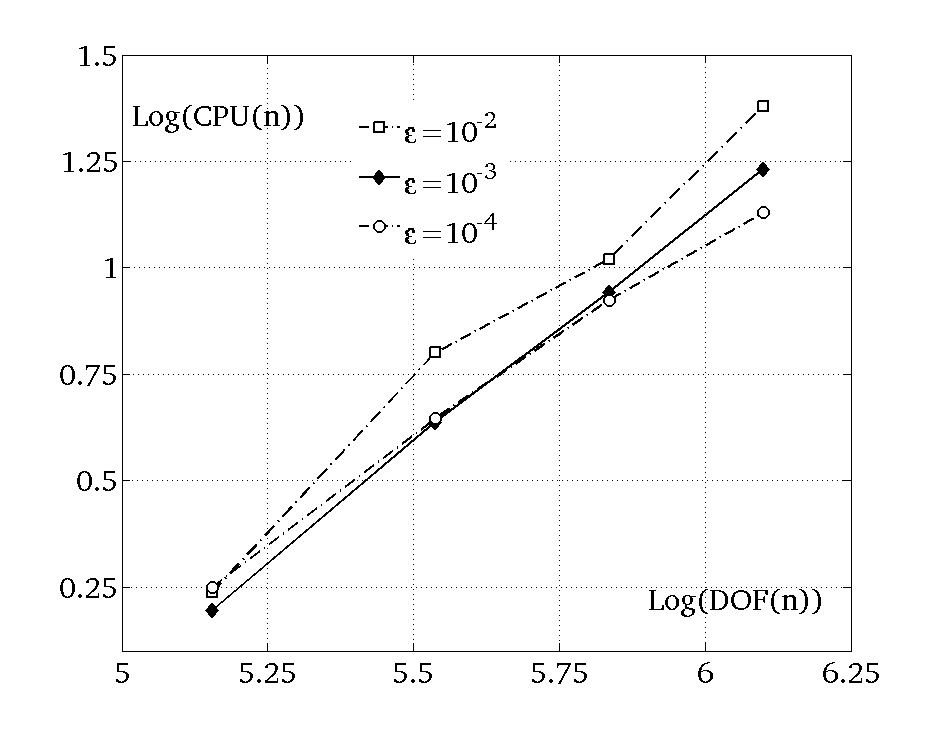}
    &
    \includegraphics[width=220pt]{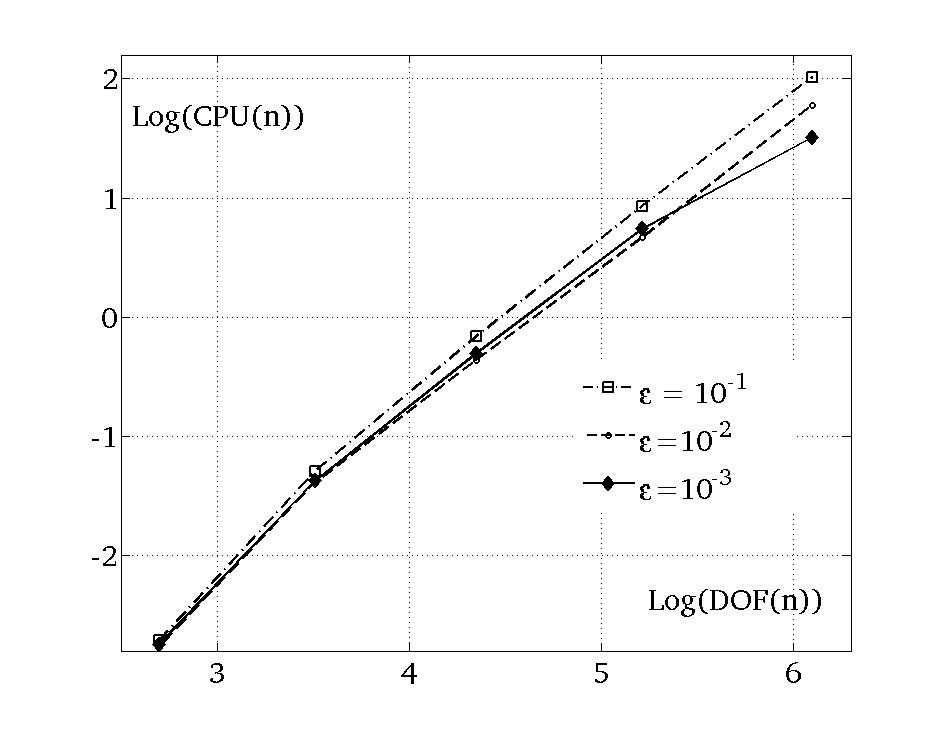}
    \\
    (a) 2D case &    (b) 3D case 
  \end{tabular}
  \caption{Cost of one inversion of $\sysmat X = Y$ in terms of CPU Time as a function of the problem size in (decimal) Log/Log scale.  
  }
  \label{fig:res-cout}
\end{figure}

\vspace{5pt}
\textbf{PCG convergence rate.} The convergence rate of the residual towards 0 for the preconditioned conjugate gradient algorithm has been measured in dimension 2 and 3 for the accuracy parameter set to $\epsilon=10^{-3}$.
The (decimal) logarithm of the residual has been plotted as a function of the iteration number on Fig. \ref{fig:pcg-conv} for the four considered meshes in dimension 2 and for 3 meshes in dimension 3.
Due to the very small number of iterations needed, this convergence rate obviously is quite large. 
\\
In dimension 3, for the finest mesh $\Tau_5$ with 1 250 000 vertices, the residual is divided by more than 150 at step one and by more than 75 at step 2. For the two other meshes, each PCG iteration divides the residual by at least 100.
\\
In dimension 2, for all four meshes $\log(\text{residual})$ displays the same global slope with respect to the number of iterations that is equal to 1.6. Globally the residual is divided by 40 at each time step.
More precisely the residual is usually divided by 100 at the first step, by 30 at the second one and by 20 at the third one.
\begin{figure}[!ht]
  \centering
  \begin{tabular}{ccc}
    \includegraphics[width=220pt]{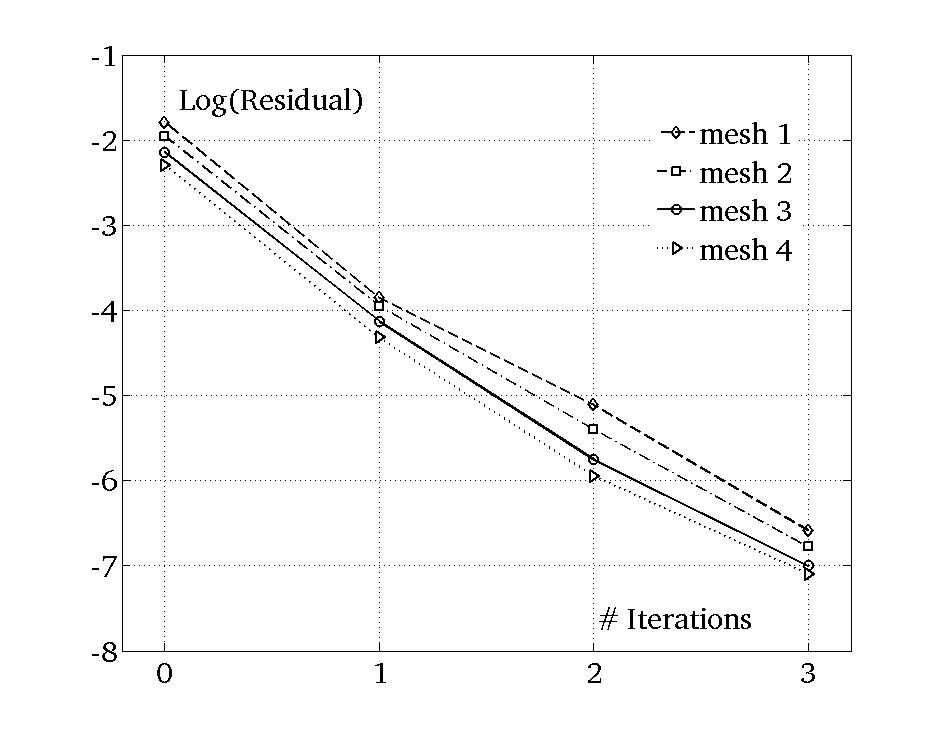} 
    &
    \includegraphics[width=220pt]{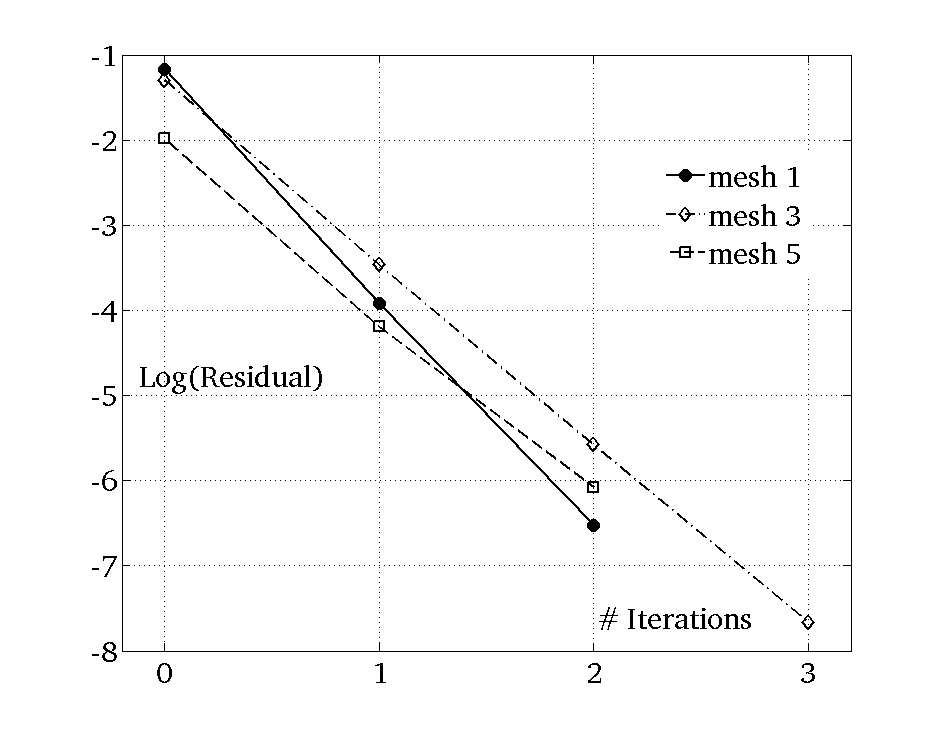} 
    \\
    (a) 2D case &    (b) 3D case 
  \end{tabular}
  \caption{PCG convergence rate. Convergence of the residual of the preconditioned system \eqref{eq:system} as a function of the number of iterations. 
On both the 2D and the 3D cases, the preconditioner is set with $\epsilon=10^{-3}$. Left, 2D case: convergence is shown for each of the four 2d  meshes. Right, 3D case:  convergence is depicted for the coarsest mesh (mesh 1), for the finest mesh (mesh 5) and on the intermediate mesh 3.
  }
  \label{fig:pcg-conv}
\end{figure}

\vspace{5pt}
\textbf{Cost calibration and profiling.}
Neither the CPU time nor the number of iterations actually  provides an absolute evaluation for the preconditioning cost in the following sense. CPU time measurements are device dependent and the iteration number does not take into account the cost for the inversion of $P_1$ and $P_K$ that may be large. 
These indicators are relevant and sufficient to evaluate the asymptotic complexity with DOF($n$) but do not allow practical comparison with other techniques.

To address this question we proceed as follows. Firstly we consider the complete algorithm profiling: we measure the amount of  time spent on each task (RHS computation, system inversion, normalisation...) at each time step and average these durations along the depolarisation sequence.
Secondly we compare the amount of time inside the PCG algorithm spent on the two predominant operations $X\mapsto P_\sysmat^{-1} X$  and $X\mapsto \sysmat X$.
The ratio between these two times provides a calibration of the preconditioner $P_\sysmat$ inversion cost in terms of matrix vector multiplication by $\sysmat$, which last operation has a fully established operational  cost.
\\
We point out that this ratio makes sense because of the almost linear complexity with DOF($n$). Practically it varies sufficiently slowly with DOF($n$) to derive a typical ratio for practically used problem size.
\\
In dimension 2 (resp. 3), these typical figures are as follows:
\begin{itemize}
\item [-] 70\% (resp. 85 \%) of the whole computational effort is dedicated on the system \eqref{eq:system} inversion,
\item [-] each operation $X\mapsto P_\sysmat^{-1} X$ has cost 15 (resp. 25) matrix-vector multiplication by $\sysmat$,
\item [-] considering an average number of iteration equal to 3, inverting $X\mapsto \sysmat^{-1}X$ has the same cost has 50  (resp. 80) matrix-vector multiplication by $\sysmat$.
\end{itemize}
\subsection{Conclusion}
\label{sec:conc}
We introduced in this paper a new preconditioning for the bidomain model based on an algebraic block-$LU$ decomposition of its system matrix $\sysmat$ and a heuristic approximation.
The complexity for solving the preconditioned system $\sysmat X=Y$ with respect to the matrix size has been numerically analysed using both a 2D and a 3D test case and a hierarchical Cholesky preconditioning.
This complexity has been numerically showed to be almost linear; which is optimal in this context (see discussion in Sec. \ref{sec:prec3}).
\\[3pt]\indent
We firstly would like to recall that
the notion of complexity is not sufficient to compare algorithms in practise.
The only certainty  is that the resolution strategy presented here will \textit{asymptotically} become more efficient than a second algorithm with worst complexity (as the problem size goes to infinity).
Being fixed a problem, the second algorithm might be more efficient.
The calibration and profiling provided in Sec. \ref{sec:res} might however help towards such comparisons and especially with the preconditioning developed in \cite{nobile-precond-2009}.
Firstly the data given in this paper do not indicate almost linear complexity. Precisely, CPU time data rather fit a complexity of 1.4 with the problem size. Despite the limitations on CPU time measurements we already mentioned, it is likely that this complexity is greater than 1.
Iteration numbers also are reported (on a test case quite close to the 3D test case here on the mesh $\Tau_4$) that are of order 6 with a flexible GMRes. Flexible GMRes performs $m$ matrix-vector multiplications and preconditioner inversions per iteration with $m$ the restart number, typically of order 25. This would mean 150 matrix-vector multiplications and preconditioner inversions.
Each preconditioner inversion itself uses an $iLU(0)$ PCG: thus one  matrix-vector multiplication and one $iLU(0)$ inversion
 per iteration.
Even assuming a fast convergence of the PCG in a few steps, this may lead to a calibration of the cost in terms of matrix-vector multiplications several times larger than the one we obtained (equal to 80).
The comparison of CPU times on the same case (almost the same processor has been used for the two papers) confirms this option.
\\[3pt]\indent
We eventually would like to underline that almost linear complexity for the resolution of \eqref{eq:system} does not mean almost linear complexity for the resolution of the bidomain model.
Assuming for  simplicity a 
linear dependence for the cost on the number of nodes, this still implies an $h^{-d}$ dependence of the cost on the mesh size $h$ and with $d$ the dimension. Considering the global cost of the simulation and not only the cost of one inversion, this now leads to an $h^{-(d+1)}$ dependence of the cost on the mesh size.
For instance, considering some precision criterion $e$ based on the activation time, that is of order 1 with $h$ as established in \cite{ABKP-2010,PRB1-2010}, the complexity for the bidomain model with respect to $e$ also is of $e^{-3}$ and $e^{-4}$ in dimension 2 and 3 respectively.
\\
Thus a linear dependence of one system inversion cost on the problem size still leads to really heavy global costs for this type of problems.
\bibliographystyle{plain}
\bibliography{biblio}
\end{document}